%% file: main.tex
\documentclass[12pt]{article}
\newcommand{\cleverefoptions}{capitalise,noabbrev}
\usepackage[amsmath, cleveref]{e-jc}

\usepackage{enumitem}
\usepackage{graphicx}
\usepackage{lipsum}
\usepackage{mathtools}
\usepackage{tikz}
\usepackage{xcolor} 
\usepackage[ruled]{algorithm2e}
\usepackage{bm}
\usetikzlibrary{positioning}

\usepackage{standalone}

\usepackage[
backend=biber,
maxnames=10,
minnames=10,
style=alphabetic]{biblatex}
\renewbibmacro{in:}{}
\DeclareFieldFormat{doi}{
  \ifhyperref
    {\href{https://doi.org/#1}{\texttt{[DOI]}}}
    {\texttt{[doi]}}}
\usepackage{xurl}
\hypersetup{breaklinks=true}

\newcommand{\defn}[1]{\emph{\textbf{#1}}}

\newcommand{\N}{\mathbb{N}}
\renewcommand{\S}{\mathfrak{S}}

\newcommand{\X}{X}
\newcommand{\ssep}{:}

\newcommand{\A}[2]{\mathcal{A}(#1, #2)}
\newcommand{\Astar}[2]{\mathcal{A}^*(#1, #2)}
\newcommand{\B}[2]{\mathcal{B}(#1, #2)}
\newcommand{\Bstar}[2]{\mathcal{B}^*(#1, #2)}

\newcommand{\C}[2]{\mathcal{C}(#1, #2)}
\newcommand{\Cstar}[2]{\mathcal{C}^*(#1, #2)}
\newcommand{\CstarI}[3]{\mathcal{C}_{#1}^*(#2, #3)}
\renewcommand{\P}{\mathcal{P}}
\newcommand{\T}{\mathcal{T}}
\newcommand{\delete}{\backslash}
\newcommand{\contract}{/}

\newcommand{\reviso}[2]{\Rev_{#1,#2}}
\newcommand{\coreviso}[2]{\Corev_{#1,#2}}

\DeclareMathOperator{\Inv}{Inv}

\DeclareMathOperator{\Corev}{Corev}
\DeclareMathOperator{\Rev}{Rev}
\DeclareMathOperator{\width}{width}

\definecolor{mygray}{HTML}{ACACAC}
\definecolor{myred}{HTML}{E6194B}
\definecolor{mygreen}{HTML}{3CB44B}
\definecolor{myblue}{HTML}{4363D8}
\definecolor{mypink}{HTML}{F032E6}
\definecolor{myorange}{HTML}{F58231}
\definecolor{mypurple}{HTML}{7F00FF}
\definecolor{mybrown}{HTML}{954535}

\usetikzlibrary{cd}
\Crefname{algocf}{Algorithm}{Algorithms}

\numberwithin{equation}{section}

\title{On Enumerating Higher Bruhat Orders\\Through Deletion and Contraction}
\author{Herman Chau}
\date{\today}

\authortext{}{Department of Mathematics, University of Washington (\email{hchau@uw.edu}). The author was partially supported by NSF Grant DMS-1764012.}

\newcounter{x}
\newcounter{y}
\newcounter{z}

\newcommand\xaxis{210}
\newcommand\yaxis{-30}
\newcommand\zaxis{90}

\newcommand\topside[3]{
  \fill[fill=yellow, draw=black,shift={(\xaxis:#1)},shift={(\yaxis:#2)},
  shift={(\zaxis:#3)}] (0,0) -- (30:1) -- (0,1) --(150:1)--(0,0);
}

\newcommand\leftside[3]{
  \fill[fill=red, draw=black,shift={(\xaxis:#1)},shift={(\yaxis:#2)},
  shift={(\zaxis:#3)}] (0,0) -- (0,-1) -- (210:1) --(150:1)--(0,0);
}

\newcommand\rightside[3]{
  \fill[fill=blue, draw=black,shift={(\xaxis:#1)},shift={(\yaxis:#2)},
  shift={(\zaxis:#3)}] (0,0) -- (30:1) -- (-30:1) --(0,-1)--(0,0);
}

\newcommand\cube[3]{
  \topside{#1}{#2}{#3} \leftside{#1}{#2}{#3} \rightside{#1}{#2}{#3}
}

\newcommand\planepartition[1]{
 \setcounter{x}{-1}
  \foreach \a in {#1} {
    \addtocounter{x}{1}
    \setcounter{y}{-1}
    \foreach \b in \a {
      \addtocounter{y}{1}
      \setcounter{z}{-1}
      \foreach \c in {1,...,\b} {
        \addtocounter{z}{1}
        \cube{\value{x}}{\value{y}}{\value{z}}
      }
    }
  }
}

\begin{document}

\maketitle

\begin{abstract}
  The higher Bruhat orders $\B{n}{k}$ were introduced by Manin--Schechtman to study discriminantal hyperplane arrangements and subsequently studied by Ziegler, who connected $\B{n}{k}$ to oriented matroids. In this paper, we consider the enumeration of $\B{n}{k}$ and improve upon Balko's asymptotic lower and upper bounds on $|\B{n}{k}|$ by a factor exponential in $k$. A proof of Ziegler's formula for $|\B{n}{n-3}|$ is given and a bijection between a certain subset of $\B{n}{n-4}$ and totally symmetric plane partitions is proved. Central to our proofs are deletion and contraction operations for the higher Bruhat orders, defined in analogy with matroids. Dual higher Bruhat orders are also introduced, and we construct isomorphisms relating the higher Bruhat orders and their duals. Additionally, weaving functions are introduced to generalize Felsner's encoding of elements in $\B{n}{2}$ to all higher Bruhat orders $\B{n}{k}$.
\end{abstract}

\section{Introduction}
The higher Bruhat orders $\B{n}{k}$ are a family of partially ordered sets introduced by Manin--Schechtman to study the combinatorics and topology of discriminantal hyperplane arrangements. Specifically, the elements in $\B{n}{k}$ are related to the order in which paths connecting antipodal chambers in a discriminantal arrangement pass through the hyperplanes in the arrangement \cite[Section 2.3]{maninschechtman89}. The intersection lattice of discriminantal arrangements was further studied by Falk \cite{falk-1994}, Bayer--Brandt \cite{bayer-brandt-1997}, and Athanasiadis \cite{athanasiadis-1999}. Furthermore, Laplante-Anfossi--Williams showed a correspondence between $\B{n}{k}$ and the cup-$i$ coproducts defining Steenrod squares in cohomology \cite{laplante-anfossi-williams-2023}.

Despite the name, the higher Bruhat orders generalize the weak order on the symmetric group $\S_n$ and \emph{not} the Bruhat order. The poset $\B{n}{1}$ is precisely the weak order on $\S_n$. The poset $\B{n}{2}$ is related to primitive sorting networks \cite{knuth92}, commutation classes of reduced expressions of the longest permutation \cite{elnitsky97, tenner2006}, weakly separated set systems \cite{danilovkarzanovkoshevoy2010}, and rhombic tilings of a regular $2n$-gon \cite{escobar-pechenik-tenner-yong-2018}. In \cite{ziegler93}, Ziegler introduced a characterization of the higher Bruhat orders in terms of collections of $k$-subsets that satisfy a certain consistency property, partially ordered by single step inclusion. See \cref{section:preliminaries} for the precise definition of consistency. The poset $\C{n}{k+1}$ on consistent subsets is isomorphic to $\B{n}{k}$, and the proof of the isomorphism of the two relies on the theory of single element extensions of alternating matroids \cite{ziegler93}. 

The goal of this paper is to obtain new enumerative and asymptotic results on the higher Bruhat orders. Along the way, the dual higher Bruhat orders $\Bstar{n}{k}$ and $\Cstar{n}{k}$ are introduced. Notably, $\Bstar{n}{k}$ and $\Cstar{n}{k}$ differ from the usual notion of poset duality, which reverses the partial order on elements. Deletion and contraction on elements of $\B{n}{k}$ and $\C{n}{k}$ are defined in analogy with the corresponding operations in matroids. As one would expect, deletion and contraction are dual to each other in the higher Bruhat orders. The main structural result is the commutative diagram of isomorphisms in \eqref{equation:commutative-diagram} relating $\B{n}{k}$, $\C{n}{k+1}$ and their duals. The commutative diagram is similar to the one in Falk's note on the duality of deletion and contraction in \cite[Remark 2.6]{falk-1994}, but applied to the higher Bruhat orders instead of the generic stratum in the Grassmannian.
See \cref{section:dual-orders} for the precise construction of the isomorphisms.

\begin{theorem}
    \label{theorem:fundamental-duality}
    For all integers $1 \le k < n$, there exist poset isomorphisms $\reviso{n}{k}$, $\coreviso{n}{k}$, $\beta_{n,k}$, and $\gamma_{n,k}$ such that the following diagram commutes:
    \begin{equation}
        \label{equation:commutative-diagram}
        \begin{tikzcd}[row sep=normal, column sep=huge]
            \B{n}{k} \arrow[r, "\reviso{n}{k}"] \arrow[d, "\beta_{n,k}"] & \C{n}{k+1} \arrow[d, "\gamma_{n,k+1}"]\\
            \Bstar{n}{n-k} \arrow[r, "\coreviso{n}{n-k}"] & \Cstar{n}{n-k-1}.
        \end{tikzcd}
    \end{equation}
\end{theorem}

Little is known about the exact enumeration of the higher Bruhat orders. Since $\B{n}{1}$ is the weak order on $\S_n$, it is clear that $|\B{n}{1}| = n!$, but already there is no known closed formula for $|\B{n}{2}|$. The sequence $|\B{n}{2}|$ begins 1, 2, 8, 62, 908, $\ldots$ and can be found in the OEIS sequence \href{https://oeis.org/A006245}{A006245}. The prime factors of $|\B{n}{2}|$ grow quickly relative to $n$, so a product formula for $|\B{n}{2}|$ is unlikely. On the other hand, for $|\B{n}{n-k}|$ the following exact formulas appear in \cite{ziegler93} for $k = 0, 1, 2, 3$: $|\B{n}{n}| = 1$, $|\B{n}{n-1}| = 2$, $|\B{n}{n-2}| = 2n$, and $|\B{n}{n-3}| = 2^n + n2^{n-2}-2n$.

Given the difficulty in computing a closed formula, one might ask for asymptotics on $|\B{n}{k}|$ instead. When $k = 2$, the best upper bound to date is $\log_2 |\B{n}{2}| \le 0.6571n^2$ due to Felsner and Valtr \cite{felsnervaltr2011}, and the best lower bound to date is $0.2721n^2 \le \log_2 |\B{n}{2}|$ due to K\"uhnast, Dallant, Felsner, and Scheucher \cite{kuhnast-dallant-felsner-scheucher}. Furthermore, Balko \cite[Theorem 3]{balko2019} recently showed that for $k \ge 2$ and sufficiently large $n \gg k$, we have
\begin{equation}
\frac{n^k}{(k+1)^{4(k+1)}} \le \log_2 |\B{n}{k}| \le \frac{2^{k-1}n^k}{k!}.
\end{equation}

The following theorem improves upon Balko's asymptotic bounds. A more precise statement is given in \cref{theorem:higher-bruhat-order-asymptotics}. The proof uses a new encoding of elements in $\B{n}{k}$ termed \emph{weaving functions}, which generalize an encoding of $\B{n}{2}$ due to Felsner's \cite{felsner1997}. The same methods are also used to prove \cref{theorem:dual-higher-bruhat-order-asymptotics}, which gives similar asymptotic bounds on $|\B{n}{n-k}| = |\Bstar{n}{k}|$.
\begin{theorem}
    \label{theorem:higher-bruhat-order-asymptotics-intro}
    For every integer $k \ge 2$, there exists a constant $c_k$ such that for sufficiently large $n \gg k$, we have
    \begin{equation}
    \frac{c_kn^k}{k!(k+1)!} \le \log_2 |\B{n}{k}| \le \frac{n^k}{k!\log{2}}.
    \end{equation}
    Furthermore, the constant $c_k$ satisfies $\displaystyle \lim_{k \to \infty} \frac{c_k}{k!/\sqrt{24 \pi k}} = 1$.
\end{theorem}

Partitioning the elements of $\B{n}{k}$ by their deletion leads to the formula for $|\B{n}{n-3}|$ that appeared in \cite{ziegler93} without proof. By enumerating elements in $\B{n}{n-4}$ with a fixed deletion set, one also obtains a connection between the higher Bruhat orders and the celebrated totally symmetric plane partitions (TSPPs), which have been studied by Andrews, Paule, and Schneider \cite{andrews-paule-schneider-2005}, Mills, Robbins, and Rumsey Jr. \cite{mills-robbins-rumsey}, and Stembridge \cite{stembridge95}, among many others.
\begin{theorem}
    \label{theorem:TSPP-bijection}
    For every integer $n \ge 5$, the poset of TSPPs $\T_{n-3}$, partially ordered by inclusion, is isomorphic to the subposet of $\B{n}{n-4}$ obtained by restricting to elements whose deletion is the commutation class of the lexicographic order on $\binom{[n-1]}{k}$.
\end{theorem}

This paper is structured as follows. In Section 2, the definitions and properties of the weak order on $\mathfrak{S}_n$, the partial orders $\B{n}{k}$ and $\C{n}{k}$, and totally symmetric plane partitions are reviewed. In Section 3, the dual partial orders $\Bstar{n}{k}$ and $\Cstar{n}{k}$ are introduced and \cref{theorem:fundamental-duality} is proved. In Section 4, the deletion and contraction operations are defined, and their fundamental properties are proved. In Section 5, weaving functions are introduced, and it is proved that distinct elements of $\B{n}{k}$ have distinct weaving functions. In Section 6, the asymptotic bounds in \cref{theorem:higher-bruhat-order-asymptotics} and \cref{theorem:dual-higher-bruhat-order-asymptotics} are proved. In Section 7, the known formula for $|\B{n}{n-3}|$ is proved and the isomorphism between a subposet of $\B{n}{n-4}$ and $\T_{n-3}$ is stated and proved. In Section 8, open problems for future study are proposed.

\section{Preliminaries}
\label{section:preliminaries}

\subsection{Weak Order on the Symmetric Group}
For a positive integer $n$, let $[n]$ denote the set $\{1, 2, \ldots, n\}$. By convention, $[0] = \varnothing$. Given an unordered set $X = \{x_1, x_2, \ldots, x_n\}$, let $(x_1, x_2, \ldots, x_n)$ denote $X$ as an ordered set. When $(x_1, x_2, \ldots, x_n)$ are integers in increasing order, the notation $[x_1, x_2, \ldots, x_n]$ will be used to emphasize that the elements are in increasing order. In order to avoid ambiguity between $[n]$ as a set of positive integers and $[n]$ as an ordered set of size one, the latter will be written as the singleton set $\{n\}$. For an integer $k$ such that $1 \le k \le n$, let $\binom{[n]}{k}$ denote the collection of ordered subsets of $k$ distinct elements in $[n]$, ordered in increasing order. For example, $\binom{[4]}{2} = \{[1,2], [1,3], [1,4], [2,3], [2,4], [3,4]\}$. By convention, $\binom{[n]}{0} = \{\varnothing\}$, and if $k$ is a positive integer such that $k > n$, then $\binom{[n]}{k} = \varnothing$.  A collection $\mathcal{S}$ of sets is said to be partially ordered by \defn{inclusion} if for any two sets $X,Y \in \mathcal{S}$, $X \le Y$ if and only if $X \subseteq Y$.

The symmetric group on $[n]$ is denoted $\mathfrak{S}_n$. The one-line notation of a permutation $\rho = (\rho_1, \rho_2, \ldots, \rho_n)$ in $\mathfrak{S}_n$ denotes the permutation that maps $i \mapsto \rho_i$ for $i \in [n]$. For example, $(2,3,1) \in \mathfrak{S}_3$ is the permutation sending $1 \mapsto 2$, $2 \mapsto 3$, and $3 \mapsto 1$. For $i \in [n-1]$, the adjacent transposition $\tau_i \in \mathfrak{S}_n$ is the permutation that swaps $i$ and $i+1$ while fixing the other elements in $[n]$. The inversion set of a permutation $\rho \in \mathfrak{S}_n$ is the subset of $\binom{[n]}{2}$ given by
\[
\Inv(\rho) = \left\{[i,j] \in \binom{[n]}{2} : \text{$\rho^{-1}(i) > \rho^{-1}(j)$}\right\}.
\]
The \defn{(right) weak order} on $\mathfrak{S}_n$ is the partial order obtained by taking the transitive closure of the cover relations where $\rho \lessdot \sigma$ if and only if $\sigma = \rho \tau_i$ for some $i \in [n-1]$ such that $\rho_i < \rho_{i+1}$. One can check that if $\rho \lessdot \sigma$ and $\sigma = \rho\tau_i$, then $\Inv(\sigma) = \Inv(\rho) \cup \{[\rho_i, \rho_{i+1}]\}$. It is well-known that the weak order on $\mathfrak{S}_n$ is isomorphic to the poset of inversion sets of permutations in $\mathfrak{S}_n$, partially ordered by inclusion \cite{bjornerbrenti2005}.

\begin{definition}[\cite{ziegler93}]
    \label{definition:single-step-inclusion}
    For a collection $\mathcal{S}$ of sets, the \defn{single step inclusion} partial order on $\mathcal{S}$ is the partial order obtained by taking the transitive closure of the cover relations where $X \lessdot Y$ if and only if $Y = X \cup \{y\}$ for some $y \not\in X$.
\end{definition}

\begin{lemma}[\cite{bjornerbrenti2005,ziegler93}]
\label{lemma:weak_order_isomorphic_to_single_step_inclusion}
The weak order on $\S_n$ is isomorphic to the single step inclusion partial order on $\{\Inv(w) : w \in \S_n\}$.
\end{lemma}

\begin{lemma}[\cite{ziegler93}]
    \label{lemma:inversion_set_characterization}
    A subset $I \subseteq \binom{[n]}{2}$ is the inversion set of some permutation in $\S_n$ if and only if $I$ satisfies the following two properties:
    \begin{itemize}
        \item $(i,j), (j,k) \in I$ implies $(i,k) \in I$ for all $i < j < k \in [n]$,
        \item $(i,j), (j,k) \not\in I$ implies $(i,k) \not\in I$ for all $i < j < k \in [n]$.
    \end{itemize}
\end{lemma}

In general, the inclusion partial order and the single step inclusion partial order on a collection of sets are not isomorphic. However, the two partial orders coincide in the case of the inversion sets of permutations in $\mathfrak{S}_n$. Further details on the weak order can be found in \cite{bjornerbrenti2005}.

\subsection{Higher Bruhat Orders}

\label{subsection:higher-bruhat-order-intro}
 In this subsection, the higher Bruhat orders are reviewed, according to the definitions of Manin--Schechtman and Ziegler. For an ordered set $X = [x_1, \ldots, x_k] \in \binom{[n]}{k}$ and a sequence of positive integers $1 \le i_1 < i_2 < \cdots < i_j \le k$, the notation $X_{i_1, \ldots, i_j}$ denotes the ordered set $X \setminus \{x_{i_1}, \ldots, x_{i_j}\}$. For example, if $X = [1, 5, 6, 8, 9]$, then $X_{3,5} = [1,5,8]$. The \defn{packet} of an ordered set $X$ is the unordered set of all size $(k-1)$ subsets of $X$, denoted $P(X) = \{\X_1, \X_2, \ldots, \X_k\}$. For a total order $\rho$ on a set $\mathcal{S}$ and a subset $I \subseteq \mathcal{S}$, the \defn{restriction} $\rho|_I$ is the total order on $I$ inherited from $\rho$.
 
 The \defn{lexicographic} (lex) order on $\binom{[n]}{k}$ is the total order on $\binom{[n]}{k}$ where, for any two elements $[i_1, \ldots, i_k], [j_1, \ldots, j_k] \in \binom{[n]}{k}$, $[i_1, \ldots, i_k] < [j_1, \ldots, j_k]$ if and only if there exists $r \in [k]$ such that $i_r < j_r$, and $i_1 = j_1$, $i_2 = j_2$, $\cdots$, $i_{r-1} = j_{r-1}$. The \defn{antilexicographic} (antilex) order is the total order on $\binom{[n]}{k}$ obtained by reversing the lex order. For a packet $P(X)$, the lex order  is $(\X_k, \X_{k-1}, \ldots, \X_1)$, and the antilex order is $(\X_1, \X_2, \ldots, \X_k)$. Notice that the indices of $X_k, X_{k-1}, \ldots, X_1$ are in decreasing order for the lex order on $P(X)$ since omitting the largest element $x_k$ corresponds to the lexicographically first element $X_k = [x_1, x_2, \ldots, x_{k-1}]$. A \defn{prefix} of $P(X)$ in lex order is a (possibly empty) ordered subset of the form $(X_k, X_{k-1}, \ldots, X_i)$, and a \defn{suffix} of $P(X)$ in lex order is a (possibly empty) ordered subset of the form $(X_i, X_{i-1}, \ldots, X_1)$.

\begin{example}
    The lex order on $\binom{[4]}{2}$ is the total order
    \[
    [1,2] < [1,3] < [1,4] < [2,3] < [2,4] < [3,4],
    \]
    which is identified with the ordered sequence
    \[
    ([1,2], [1,3], [1,4], [2,3], [2,4], [3,4]).
    \]
    When unambiguous, the brackets and commas are omitted from an ordered set. Thus, the lex order on $\binom{[4]}{2}$ may be written more compactly as $(12, 13, 14, 23, 24, 34)$. The antilex order on $\binom{[4]}{2}$ is $(34, 24, 23, 14, 13, 12)$. The restriction of the antilex order on $\binom{[4]}{2}$ to the packet $P(123)$ is $(23, 13, 12)$.    
\end{example}

\begin{definition}[\cite{maninschechtman89}]
A total order $\rho$ on $\binom{[n]}{k}$ is \defn{admissible} if for every $X \in \binom{[n]}{k+1}$, the restriction $\rho|_{P(X)}$ is either the lex or antilex order on $P(X)$. The set of admissible orders on $\binom{[n]}{k}$ is denoted $\mathbf{\A{n}{k}}$. Given an admissible order $\rho \in \A{n}{k}$, its \defn{reversal set} is 
\begin{equation}
\boldsymbol{\reviso{n}{k}(\rho)} = \left\{X \in \binom{[n]}{k+1} \ssep \text{$\rho|_{P(X)}$ is the antilex order}\right\}.
\end{equation}
\end{definition}

The subscript in $\Rev_{n,k}$ is dropped when $k$ and $n$ are clear from context. For integers $1 \le k \le n$, the lex (resp. antilex) order on $\binom{[n]}{k}$ is always an admissible order in $\A{n}{k}$ since the restriction to any packet is always the lex (resp. antilex) order on the packet. The reversal set of the lex order on $\binom{[n]}{k}$ is the empty set since there are no packets in antilex order. The reversal set of the antilex order on $\binom{[n]}{k}$ is $\binom{[n]}{k+1}$ since the restriction to every packet is the antilex order. 

\begin{example}
    Consider the total order
    \begin{equation}
    \label{equation:example-admissible-order}
    \rho = (23, 24, 25, 45, 13, 15, 35, 14, 34, 12).
    \end{equation}
    The restriction of $\rho$ to the packet $P(145) = \{14, 15, 45\}$ is $\rho|_{P(145)} = (45, 15, 14)$, which is the antilex order on $P(145)$. One can check that for every $X \in \binom{[5]}{3}$, the restriction of $\rho$ to $P(X)$ is either the lex or antilex order on $P(X)$. Thus, $\rho$ is in $\A{5}{2}$. Its reversal set is
    \begin{equation}
        \label{equation:example-reversal-set}
        \Rev(\rho) = \{123, 124, 125, 145, 345\}.
    \end{equation}
\end{example}

\begin{example}
    As a nonexample, consider transposing $13$ and $15$ in the total order $\rho$ of \eqref{equation:example-admissible-order}. This yields the total order
    \begin{equation}
        \tau = (23, 24, 25, 45, 15, 13, 35, 14, 34, 12).
    \end{equation}
    The restriction of $\tau$ to the packet $P(135) = \{13, 15, 35\}$ is $\tau|_{P(135)} = (15, 13, 35)$, which is neither the lex order nor the antilex order on $P(135)$. Thus, $\tau$ is \emph{not} an admissible order.

    Note that $P(13) \cap P(15) \neq \varnothing$. However, $P(45) \cap P(13) = \varnothing$ and one can check that transposing $45$ and $13$ in $\rho$ does yield an admissible total order. 
\end{example}

Subsets $X,Y \in \binom{[n]}{k}$ are said to \defn{commute} if $P(X) \cap P(Y) = \varnothing$. If $\rho = (\rho_1, \ldots, \rho_{\binom{n}{k}}) \in \A{n}{k}$ and $\rho_i, \rho_{i+1}$ commute, then one can observe that the total order obtained by transposing $\rho_i$ and $\rho_{i+1}$ is also an admissible order. Two admissible orders $\rho, \rho' \in \A{n}{k}$ are \defn{commutation equivalent}, denoted $\boldsymbol{\rho \sim \rho'}$, if $\rho'$ can be obtained from $\rho$ by a sequence of commutations. The \defn{commutation class} of an admissible order $\rho \in \A{n}{k}$, denoted $\boldsymbol{[\rho]}$, is the set of admissible orders that are commutation equivalent to $\rho$. Note that reversal sets are invariant under commutations, so $\Rev([\rho])$ is well-defined for a commutation class $[\rho]$.

Let $\rho = (\rho_1, \ldots, \rho_{\binom{n}{k}}) \in \A{n}{k}$ and $X \in \binom{[n]}{k+1}$, such that $\rho|_{P(X)}$ is a saturated chain $(\rho_i, \rho_{i+1}, \ldots, \rho_{i+k})$ of $\rho$. Reversing the order on the saturated chain yields a new total order $\sigma = (\rho_1, \ldots, \rho_{i-1}, \rho_{i+k}, \ldots, \rho_i, \rho_{i+k+1}, \ldots, \rho_{\binom{n}{k}})$. If $\rho|_{P(X)}$ is the lex (resp. antilex) order on $P(X)$, then $\sigma$ is said to be obtained from $\rho$ by a \defn{packet flip of $P(X)$ from lex to antilex} (resp. antilex to lex). One can observe that the new total order $\sigma$ is also an admissible order.
\begin{definition}[\cite{maninschechtman89}]
    \label{definition:higher-bruhat-order-1}
    For integers $1 \le k \le n$, the \defn{higher Bruhat order} $\mathbf{\B{n}{k}}$ is the poset on commutation classes $\A{n}{k}/\hspace{-0.2em}\sim$ obtained by taking the transitive closure of the cover relations where $[\rho] \lessdot [\sigma]$ if and only if there exist $\rho' \in [\rho]$ and $\sigma' \in [\sigma]$ such that $\sigma'$ is obtained from $\rho'$ by a packet flip from lex to antilex.
\end{definition}

\begin{example}
\label{example:total-order-reversal-set}
 The total order $\rho = (23, 13, 24, 14, 12, 34)$ is an admissible order in $\A{4}{2}$ with reversal set $\Rev(\rho) = \{123, 124\}$. The commutation class $[\rho]$ is
 \[
    \{(23, 24, 13, 14, 12, 34), (23, 13, 24, 14, 34, 12),
    (23, 24, 13, 14, 34, 12), (23, 13, 24, 14, 12, 34)\},
 \]
 obtained by commuting only $(13, 24)$, only $(12, 34)$, both, or neither. The admissible order $\sigma = (23, 24, 34, 14, 13, 12)$ is obtained from $(23, 24, 13, 14, 34, 12) \in [\rho]$ by a packet flip of $P(134) = \{13, 14, 34\}$ from lex to antilex. Thus, $[\rho] \lessdot [\sigma]$ in $\B{4}{2}$. The Hasse diagram of $\B{4}{2}$ is depicted on the left in \cref{fig:b-4-2} with the corresponding reversal sets on the right, ordered by single step inclusion.

 \begin{figure}[!ht]
     \centering
\scalebox{1.0}{
            \begin{tikzpicture}[every node/.style={font=\small}]
            \node (a1) at (0, 2.5) {$[(34, 24, 23, 14, 13, 12)]$};
            \node (a2) at (-2,1.25) {$[(23, 24, 34, 14, 13, 12)]$};
            \node (a3) at (2,1.25) {$[(34, 24, 14, 12, 13, 23)]$};
            \node (a4) at (-2,0) {$[(23, 13, 24, 14, 12, 34)]$};
            \node (a5) at (2,0) {$[(12, 34, 14, 13, 24, 23)]$};
            \node (a6) at (-2,-1.25) {$[(23, 13, 12, 14, 24, 34)]$};
            \node (a7) at (2,-1.25) {$[(12, 13, 14, 34, 24, 23)]$};
            \node (a8) at (0,-2.5) {$[(12, 13, 14, 23, 24, 34)]$};
            
            \draw (a1) -- (a2);
            \draw (a1) -- (a3);
            \draw (a2) -- (a4);
            \draw (a3) -- (a5);
            \draw (a4) -- (a6);
            \draw (a5) -- (a7);
            \draw (a6) -- (a8);
            \draw (a7) -- (a8);
            \end{tikzpicture}
        }
        \scalebox{1.0}{
            \begin{tikzpicture}[every node/.style={font=\small}]
            \node (a1) at (0, 2.5) {$\{123, 124, 134, 234\}$};
            \node (a2) at (-2,1.25) {$\{123, 124, 134\}$};
            \node (a3) at (2,1.25) {$\{124, 134, 234\}$};
            \node (a4) at (-2,0) {$\{123, 124\}$};
            \node (a5) at (2,0) {$\{134, 234\}$};
            \node (a6) at (-2,-1.25) {$\{123\}$};
            \node (a7) at (2,-1.25) {$\{234\}$};
            \node (a8) at (0,-2.5) {$\varnothing$};
            
            \draw (a1) -- (a2);
            \draw (a1) -- (a3);
            \draw (a2) -- (a4);
            \draw (a3) -- (a5);
            \draw (a4) -- (a6);
            \draw (a5) -- (a7);
            \draw (a6) -- (a8);
            \draw (a7) -- (a8);
            \end{tikzpicture}
        }

    \caption{The partial orders $\B{4}{2}$ on the left and $\C{4}{3}$ on the right.}
     \label{fig:b-4-2}
 \end{figure}
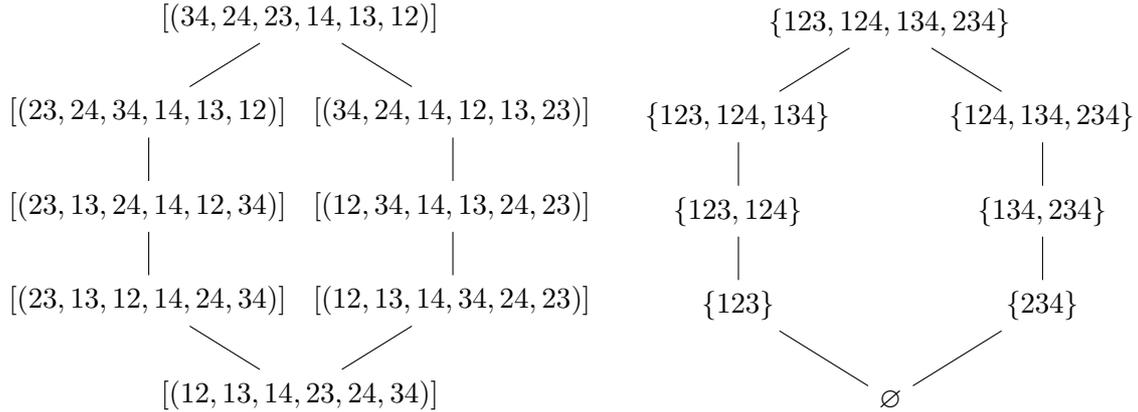
\end{example}

\begin{definition}[\cite{ziegler93}]
For integers $1 \le k \le n$, a subset $I \subseteq \binom{[n]}{k}$ is \defn{consistent} if for every $X \in \binom{[n]}{k+1}$, the intersection $P(X) \cap I$ is either a prefix or a suffix of $P(X)$ in lex order. The poset of all consistent subsets of $\binom{[n]}{k}$ ordered by single step inclusion is denoted $\boldsymbol{\C{n}{k}}$.
\end{definition}

\begin{example}
    One can check that the reversal set $\Rev(\rho)$ in \eqref{equation:example-reversal-set} is a consistent subset in $\C{5}{3}$ by considering $\Rev(\rho) \cap P(X)$ for all $X \in \binom{[5]}{4}$. For example, the intersection $\Rev(\rho) \cap P(1234) = (123, 124)$ is a prefix of $P(1234)$ in lex order, and the intersection $\Rev(\rho) \cap P(1345) = (145, 345)$ is a suffix of $P(1345)$ in lex order. The remaining 3 subsets are left to the reader.
\end{example}

Total orders on $\binom{[n]}{1}$ are easily identified with permutations in $\S_n$. Under this identification, each permutation is in its own commutation class, and packet flips correspond to adjacent transpositions. Furthermore, reversal sets of total orders of $\binom{[n]}{1}$ correspond to inversion sets of permutations in $\S_n$, which are the consistent subsets of $\binom{[n]}{2}$. One can check from the definitions that $\B{n}{1}$ is isomorphic to $\S_n$ under the weak order. By \cref{lemma:weak_order_isomorphic_to_single_step_inclusion} and \cref{lemma:inversion_set_characterization}, $\C{n}{2}$ is also isomorphic to $\S_n$ under the weak order, and hence $\B{n}{1}$ and $\C{n}{2}$ are isomorphic as posets. This isomorphism motivated the following theorem due to Ziegler.

\begin{theorem}[{\cite[Theorem~4.1]{ziegler93}}]
    \label{theorem:higher-bruhat-order-isomorphism}
    For integers $1 \le k < n$, the map $\reviso{n}{k}$ is a poset isomorphism between $\B{n}{k}$ and $\C{n}{k+1}$.
\end{theorem}

\begin{example}
    The Hasse diagram of $\C{4}{3}$ is drawn on the right in \cref{fig:b-4-2}. Per \cref{theorem:higher-bruhat-order-isomorphism}, $\B{4}{2}$ and $\C{4}{3}$ are isomorphic as posets, with the isomorphism sending $[\rho] \mapsto \Rev([\rho])$, which can be seen from the figure.
\end{example}

\subsection{Totally Symmetric Plane Partitions}
For positive integers $r$, $s$, and $t$, a \defn{plane partition} that fits in an $r \times s \times t$ box is a subset $T$ of the Cartesian product $[r] \times [s] \times [t]$ such that, if $(x_1, x_2, x_3) \in T$, then $[x_1] \times [x_2] \times [x_3] \subseteq T$. For plane partitions, the inclusion partial order and single step inclusion partial order are isomorphic. The number of plane partitions that fit inside an $r \times s \times t$ box is given by a celebrated product formula due to MacMahon \cite{macmahon-1960}
\begin{equation}
    \label{equation:macmahon-product-formula}
    \prod_{1 \le i,j,k \le n} \frac{i+j+k-1}{i+j+k-2}.
\end{equation}

\begin{definition}
    For a positive integer $n$, a \defn{totally symmetric plane partition} (TSPP) that fits in an $n \times n \times n$ box is a plane partition $T \subseteq [n]^3$ such that, for all $\sigma \in \mathfrak{S}_3$, if $(x_1,x_2,x_3) \in T$, then $(x_{\sigma(1)}, x_{\sigma(2)}, x_{\sigma(3)}) \in T$. 
\end{definition}

Let $\T_n$ denote the set of all TSPPs that fit in an $n \times n \times n$ box.
The number of TSPPs $|\T_n|$ is given by a product formula analogous to \eqref{equation:macmahon-product-formula} due to Stembridge \cite{stembridge95}
\begin{equation}
    \label{equation:stembridge-product-formula}
    |\T_n| = \prod_{1 \le i \le j \le k \le n} \frac{i+j+k-1}{i+j+k-2}.
\end{equation}
Note that the two product formulas differ only in that $i,j,k$ are weakly increasing in \eqref{equation:stembridge-product-formula}.

\begin{example}
    \label{example:TSPP}
    The set $T = \{(1,1,1), (1,1,2), (1,1,3), (1,2,1), (1,1,3), (2,1,1), (3,1,1)\}$ is a totally symmetric plane partition that fits in a $3 \times 3 \times 3$ box. One can visualize $T$ as a set of boxes in $\mathbb{R}^3$ as in \cref{figure:TSPP example}.

    \begin{figure}[!ht]
        \centering
        \begin{tikzpicture}[scale=0.8]
        \planepartition{{3,1,1},{1},{1}}
        \end{tikzpicture}
        \caption{The TSPP in \cref{example:TSPP}.}
        \label{figure:TSPP example}
    \end{figure}
\end{example}

\section{Dual Higher Bruhat Orders}

\label{section:dual-orders}
In \cref{subsection:higher-bruhat-order-intro}, the definitions of admissible orders in $\A{n}{k}$ or consistent sets in $\C{n}{k}$ involved packets $P(X)$ for sets $X \in \binom{[n]}{k+1}$ of cardinality $k+1$. By considering sets $X \in \binom{[n]}{k-1}$ of cardinality $k-1$, one is led to develop the notion of the dual higher Bruhat order. In this subsection, unless otherwise specified, $k$ and $n$ are fixed positive integers that satisfy $1 \le k \le n$.

\begin{definition}
    For $X \in \binom{[n]}{k}$, the \defn{copacket} of $X$ is $P^*_n(X) = \{X \cup \{i\} : i \in [n] \setminus X\}$.
\end{definition}

Unlike $P(X)$, the copacket is dependent on the ambient set $[n]$. When it is clear from context, the subscript $n$ will be suppressed, so that $P^*(X)$ denotes the copacket of $X$ with respect to $[n]$. The $i$th element of $P^*(X)$ in lex order will be written as $X^i$.

\begin{lemma}
    \label{lemma:commute-cocommute-coincide}
    For $X,Y \in \binom{[n]}{k}$, $P(X) \cap P(Y) = \varnothing$ if and only if $P^*(X) \cap P^*(Y) = \varnothing$.
\end{lemma}
\begin{proof}
    It suffices to note that $P(X) \cap P(Y) = \varnothing$ and $P^*(X) \cap P^*(Y) = \varnothing$ are both equivalent to the condition $|X \setminus Y| > 1$ and $|Y \setminus X| > 1$. 
\end{proof}

Recall that two elements $X,Y \in \binom{[n]}{k}$ commute if $P(X) \cap P(Y) = \varnothing$. It would be natural to say that $X,Y$ cocommute if $P^*(X) \cap P^*(Y) = \varnothing$, but \cref{lemma:commute-cocommute-coincide} implies that $X,Y$ cocommute if and only if $X,Y$ commute. Henceforth, even in the context of duality, elements will be said to commute since the definitions of cocommutation and commutation coincide.

\begin{definition}
    \label{definition:coadmissible-order}
    A total order $\rho$ on $\binom{[n]}{k}$ is \defn{coadmissible} if the restriction $\rho|_{P^*(X)}$ is the lex or antilex order on $P^*(X)$ for every $X \in \binom{[n]}{k-1}$. The set of all coadmissible orders on $\binom{[n]}{k}$ is denoted $\mathbf{\Astar{n}{k}}$. The commutation class of a coadmissible order $\rho$ is written $\boldsymbol{[\rho]}$.
\end{definition}

It is not difficult to see that if $\rho \in \Astar{n}{k}$ is a coadmissible order, then transposing adjacent elements that cocommute yields a new coadmissible order. For $X \in \binom{[n]}{k-1}$, if the copacket $P^*(X) = (X^1, X^2, \ldots, X^{n-(k-1)})$ occurs as a saturated chain in $\rho$, then reversing the order of $P^*(X)$ in $\rho$ yields a new coadmissble order said to be obtained by a \defn{copacket flip of $P^*(X)$ from lex to antilex}.

\begin{definition}
    The \defn{dual higher Bruhat order} $\Bstar{n}{k}$ is the partial order on commutation classes of elements in $\Astar{n}{k}$, where $[\rho] \lessdot [\sigma]$ if and only if there exists $\rho' \in [\rho]$ and $\sigma' \in [\sigma]$  such that $\sigma'$ is obtained from $\rho'$ by a copacket flip from lex order to antilex order.
\end{definition}

\begin{definition}
    \label{definition:coreversal-set}
    The \defn{coreversal set} of a coadmissible order $\rho \in \Astar{n}{k}$ is the set
    \[
    \coreviso{n}{k}(\rho) = \left\{X \in \binom{[n]}{k-1} : \text{$\rho|_{P^*(X)}$ is in antilex order}\right\}.
    \]
    When $n$ and $k$ are clear from context, the subscript is omitted in $\coreviso{n}{k}$.
\end{definition}

\begin{definition}
    A set $I \subseteq \binom{[n]}{k}$ is \defn{coconsistent} if for all $X \in \binom{[n]}{k-1}$, $P^*(X) \cap I$ is a prefix or suffix of $P^*(X)$ in lex order. The poset of coconsistent subsets of $\binom{[n]}{k}$ partially ordered by single step inclusion is denoted $\Cstar{n}{k}$.
\end{definition}

\begin{example}
    For $k = 2$ and $n = 4$, consider the total order $\rho = (12, 34, 23, 13, 24, 14)$ on $\binom{[4]}{2}$. The restriction of $\rho$ to the copacket $P^*(\{3\}) = \{13, 23, 34\}$ is the antilex order $(34, 23, 13)$. One can check that the restriction of $\rho$ to the other copackets $P^*(\{1\})$, $P^*(\{2\})$, and $P^*(\{4\})$ are all either the lex order or the antilex order. Thus, $\rho$ is a coadmissible order in $\Astar{4}{2}$. The coreversal set of $\rho$ is $\Corev(\rho) = \{\{3\}, \{4\}\}$. Observe that $\Corev(\rho)$ is a coconsistent set in $\Cstar{4}{1}$ since the intersection $\Corev(\rho) \cap P^*(\varnothing)$ is $\{\{3\},\{4\}\}$, which is a suffix of the lex order on $P^*(\varnothing)$. This is the only intersection one needs to check since $\binom{[4]}{0} = \{\varnothing\}$.
\end{example}

Recall that by \cref{theorem:higher-bruhat-order-isomorphism}, the map $\reviso{n}{k}$ is a well-defined isomorphism between $\B{n}{k}$ and $\C{n}{k+1}$. Before proving \cref{theorem:fundamental-duality}, the maps $\beta_{n,k}$, and $\gamma_{n,k}$ need to be defined. The proof that $\coreviso{n}{k}$, $\beta_{n,k}$, and $\gamma_{n,k}$ are well-defined will be deferred until the proof of \cref{lemma:well-defined-isomorphisms}.
Define the maps
\begin{equation}
    \label{equation:beta_gamma_maps}
    \begin{aligned}
        \beta_{n,k}&: \B{n}{k} \to \Bstar{n}{n-k} &&\text{ where } [(\rho_1, \ldots, \rho_{\binom{n}{k}})] \mapsto [([n] \setminus \rho_{\binom{n}{k}}, \ldots, [n] \setminus \rho_1)],\\
    \gamma_{n,k}&: \C{n}{k} \to \Cstar{n}{n-k} &&\text{ where }I \mapsto \{[n] \setminus X : X \in I\}.\\
    \end{aligned}
\end{equation}
For a commutation class $[\rho] \in \B{n}{k}$, the notation $[\rho]^*$ is used to denote $\beta_{n,k}([\rho]) \in \Bstar{n}{n-k}$. Similarly, for a consistent set $I \in \C{n}{k}$, the notation $I^*$ is used to denote $\gamma_{n,k}(I) \in \Cstar{n}{n-k}$.

\begin{example}
    Let $\rho \in \A{5}{2}$ be the admissible order in \eqref{equation:example-admissible-order}, with reversal set in \eqref{equation:example-reversal-set}. The commutation class $\beta_{5,2}([\rho])$ is
    \begin{equation}
        \label{equation:example-coadmissible-order}
        \beta_{5,2}([\rho]) = [(345, 125, 235, 124, 234, 245, 123, 134, 135, 145)] \in \Bstar{5}{3},
    \end{equation}
    and the coreversal set $\coreviso{5}{3}(\beta_{5,2}([\rho]))$ is
    \begin{equation}
        \coreviso{5}{3}(\beta_{5,2}([\rho])) = \{45, 35, 34, 23, 12\}.
    \end{equation}
    One can check that the set $\gamma_{5,3}(\Rev([\rho]))$ is coconsistent and equal to $\coreviso{5}{3}(\beta_{5,2}([\rho]))$.
\end{example}

\begin{example}
    \label{example:b-4-2-and-dual}
    The map $\beta_{4,2}$ is an isomorphism between $\B{4}{2}$ and $\Bstar{4}{2}$. The Hasse diagram of $\B{4}{2}$ is on the left in \cref{fig:b-4-2} and the Hasse diagram of $\Bstar{4}{2}$ is on the left in \cref{figure:dual-higher-bruhat-orders}. The map $\gamma_{4,3}$ is an isomorphism between $\C{4}{3}$ and $\Cstar{4}{1}$. The Hasse diagram of $\C{4}{3}$ is on the right in \cref{fig:b-4-2}, and the Hasse diagram of $\C{4}{1}$ is on the right in \cref{figure:dual-higher-bruhat-orders}. Note that although the elements of $\B{4}{2}$ and $\Bstar{4}{2}$ are both commutation classes of total orders on $\binom{[4]}{2}$, some total orders such as $(23, 13, 24, 14, 12, 34)$ are admissible but not coadmissible, and some total orders such as $(12, 34, 23, 13, 24, 14)$ are coadmissible but not admissible.
    
    \begin{figure}[!ht]
        \centering
        \scalebox{1.0}{
            \begin{tikzpicture}[every node/.style={font=\small}]
                    
            \node (b1) at (0, 2.5) {$[(34, 24, 23, 14, 13, 12)]$};
            \node (b2) at (-2,1.25) {$[(34, 24, 23, 12, 13, 14)]$};
            \node (b3) at (2,1.25) {$[(14, 24, 34, 23, 13, 12)]$};
            \node (b4) at (-2,0) {$[(12,34,23,13,24,14)]$};
            \node (b5) at (2,0) {$[(14, 13, 24, 23, 12, 34)]$};
            \node (b6) at (-2,-1.25) {$[(12, 13, 23, 34, 24, 14)]$};
            \node (b7) at (2,-1.25) {$[(14, 13, 12, 23, 24, 34)]$};
            \node (b8) at (0,-2.5) {$[(12, 13, 14, 23, 24, 34)]$};
            
            \draw (b1) -- (b2);
            \draw (b1) -- (b3);
            \draw (b2) -- (b4);
            \draw (b3) -- (b5);
            \draw (b4) -- (b6);
            \draw (b5) -- (b7);
            \draw (b6) -- (b8);
            \draw (b7) -- (b8);
    
            \end{tikzpicture}
        }
        \scalebox{1.0}{
            \begin{tikzpicture}[every node/.style={font=\small}]
            \node (a1) at (0, 2.5) {$\{\{1\},\{2\},\{3\},\{4\}\}$};
            \node (a2) at (-2,1.25) {$\{\{2\},\{3\},\{4\}\}$};
            \node (a3) at (2,1.25) {$\{\{1\},\{2\},\{3\}\}$};
            \node (a4) at (-2,0) {$\{\{3\},\{4\}\}$};
            \node (a5) at (2,0) {$\{\{1\},\{2\}\}$};
            \node (a6) at (-2,-1.25) {$\{\{4\}\}$};
            \node (a7) at (2,-1.25) {$\{\{1\}\}$};
            \node (a8) at (0,-2.5) {$\varnothing$};
            
            \draw (a1) -- (a2);
            \draw (a1) -- (a3);
            \draw (a2) -- (a4);
            \draw (a3) -- (a5);
            \draw (a4) -- (a6);
            \draw (a5) -- (a7);
            \draw (a6) -- (a8);
            \draw (a7) -- (a8);
            \end{tikzpicture}
        }
        \caption{The partial orders $\Bstar{4}{2}$  on the left and $\Cstar{4}{1}$ on the right.}
        \label{figure:dual-higher-bruhat-orders}
    \end{figure}
\end{example}

\begin{lemma}
    \label{lemma:well-defined-isomorphisms}
    The maps $\coreviso{n}{k}$, $\beta_{n,k}$, and $\gamma_{n,k}$ are well-defined.
\end{lemma}
\begin{proof}
    First, consider $\coreviso{n}{k}$. For a commutation class $[\rho] \in \Bstar{n}{k}$, one can readily check from \cref{definition:coadmissible-order} and \cref{definition:coreversal-set} that $\coreviso{n}{k}(\rho') = \coreviso{n}{k}(\rho)$ for all $\rho' \in [\rho]$. Thus, the coreversal set is well-defined on commutation classes. It remains to check that the image of $\coreviso{n}{k}$ lies in $\Cstar{n}{k-1}$. To see this, let $[\rho] \in \Bstar{n}{k}$, $X \in \binom{[n]}{k-2}$, and $[n] \setminus X = [y_1 < y_2 < \cdots < y_{n-(k-2)}]$. For integers $r, s, t \in [n-(k-2)]$ with $r < s < t$, one can write $X^r$, $X^s$ and $X^t$ as
    \begin{align*}
        X^r &= X \cup \{y_r\},\\
        X^s &= X \cup \{y_s\},\text{ and}\\
        X^t &= X \cup \{y_t\}.
    \end{align*}
    Now suppose $X^r, X^t \in \coreviso{n}{k}([\rho])$. Then the restriction of $\rho$ to $P^*(X^r)$ or $P^*(X^t)$ is the antilex order on each copacket respectively. Therefore, in $\rho$
    \begin{equation}
       X^s \cup \{y_t\} = X^t \cup \{y_s\} < X^t \cup \{y_r\} = X^r \cup \{y_t\} < X^r \cup \{y_s\} = X^s \cup \{y_r\}. 
    \end{equation}
    Since $X^s \cup \{y_t\} < X^s \cup \{y_r\}$ in $\rho$, it follows that the restriction of $\rho$ to $P^*(X^s)$ is the antilex order. Thus, if $X^r, X^t \in \coreviso{n}{k}([\rho])$ then $X^s \in \coreviso{n}{k}([\rho])$ also. A similar argument shows that if $X^r, X^t \not\in \coreviso{n}{k}([\rho])$, then $X^t \not\in \coreviso{n}{k}([\rho])$. Therefore, $\coreviso{n}{k}([\rho]) \subseteq \binom{[n]}{k-1}$ is coconsistent, so $\coreviso{n}{k}([\rho]) \in \Cstar{n}{k-1}$.

    Next, consider $\beta_{n,k}$. For a total order $\rho$ on $\binom{[n]}{k}$, the map sending $(\rho_1, \ldots, \rho_{\binom{n}{k}})$ to $([n] \setminus \rho_{\binom{n}{k}}, \ldots, [n] \setminus \rho_1)$ complements and reverses $\rho$. Thus, if $\rho, \rho' \in \A{n}{k}$ differ by a commutation move, then the total orders on $\binom{[n]}{n-k}$ obtained by complementing and reversing $\rho$ and $\rho'$ also differ by a commutation move. Therefore, the commutation class $\beta_{n,k}([\rho])$ as defined in \cref{equation:beta_gamma_maps} is independent of the choice of representative in $[\rho]$. Furthermore, if $\rho$ is admissible, then for any $X \in \binom{[n]}{k+1}$, the restriction $\rho|_{P(X)}$ is the lex (resp. antilex) order on $P(X)$ so complementation and reversal of the restriction yields the lex (resp. antilex) order on $P^*([n] \setminus X)$. Thus, if $\rho$ is an admissible order, then $([n] \setminus \rho_{\binom{n}{k}}, \ldots, [n] \setminus \rho_1)$ is a coadmissible order. Therefore, $\beta_{n,k}([\rho]) \in \Bstar{n}{n-k}$ is a commutation class of coadmissible orders. 

    Finally, consider $\gamma_{n,k}$. For $I \in \C{n}{k}$ and $X \in \binom{[n]}{k+1}$, the intersection $P(X) \cap I$ is a prefix or suffix of $P(X)$ in lex order. One can check that this implies that $P^*([n] \setminus X) \cap \gamma_{n,k}(I)$ is a prefix or suffix of $P^*([n] \setminus X)$ in lex order. Since the intersection of $\gamma_{n,k}(I)$ with any copacket $P^*([n] \setminus X)$ is a prefix or suffix, it follows that $\gamma_{n,k}(I) \in \Cstar{n}{n-k}$.
\end{proof}

\begin{proof}[Proof of \cref{theorem:fundamental-duality}]
    It suffices to prove that $\beta_{n,k}$ and $\gamma_{n,k}$ are poset isomorphisms and that $\coreviso{n}{n-k} = \gamma_{n,k+1} \circ \reviso{n}{k} \circ \beta_{n,k}^{-1}$. It then follows that the diagram commutes and that $\coreviso{n}{n-k}$ is also a poset isomorphism.
    
    Let $[\rho], [\sigma] \in \B{n}{k}$, and suppose $[\rho] \lessdot [\sigma]$, where $\sigma$ is obtained from $\rho$ by a packet flip of $P(X)$ from lex to antilex for some $X \in \binom{[n]}{k+1}$. The packet $P(X)$ bijects with the copacket $P^*([n] \setminus X)$ by taking complements 
    \begin{equation}
    \label{equation:packet-copacket-bijection}
    X_i \leftrightarrow [n] \setminus X_i = ([n] \setminus X)^i.
    \end{equation}
    Under \eqref{equation:packet-copacket-bijection}, the packet $P(X)$ in lex order bijects to $P^*([n] \setminus X)$ in antilex order. Since $\beta_{n,k}$ complements each set and then reverses the total order, $\beta_{n,k}([\sigma])$ is obtained from $\beta_{n,k}([\rho])$ by a copacket flip of $P^*([n] \setminus X)$ from lex to antilex. Thus, $\beta_{n,k}([\rho]) \lessdot \beta_{n,k}([\sigma])$ in $\Bstar{n}{n-k}$. The inverse map $\beta_{n,k}^{-1}$ is also defined by complementation and reversal, sending
    \begin{equation}
        [(\rho_1, \ldots, \rho_{\binom{n}{n-k}})] \mapsto [([n] \setminus \rho_{\binom{n}{n-k}}, \ldots, [n] \setminus \rho_1)]
    \end{equation}
    for $[\rho] = [(\rho_1, \ldots, \rho_{\binom{n}{n-k}})] \in \Bstar{n}{n-k}$. By a similar argument, $\beta_{n,k}^{-1}$ also preserves covering relations. Thus, $\beta_{n,k}$ is a poset isomorphism between $\B{n}{k}$ and $\Bstar{n}{n-k}$.

    Next, to show that $\gamma_{n,k}$ is a poset isomorphism, Let $I,J \in \C{n}{k}$ such that $I \lessdot J$. Since the partial order on $\C{n}{k}$ is single step inclusion, $J = I \cup \{X\}$ for some $X \in \binom{[n]}{k} \setminus I$. Then $\gamma_{n,k}(J) = \gamma_{n,k}(I) \cup \{[n] \setminus X\}$. The partial order on $\Cstar{n}{n-k}$ is also single step inclusion, so $\gamma_{n,k}(I) \lessdot \gamma_{n,k}(J)$ in $\Cstar{n}{n-k}$. The inverse map $\gamma_{n,k}^{-1}$ is defined by sending $I \in \Cstar{n}{n-k}$ to $\{[n] \setminus X : X \in I\}$ and also preserves covering relations by a similar argument. Thus, $\gamma_{n,k}$ is a poset isomorphism between $\C{n}{k}$ and $\Cstar{n}{n-k}$.

    Finally, to prove that $\coreviso{n}{n-k} = \gamma_{n,k+1} \circ \reviso{n}{k} \circ \beta_{n,k}^{-1}$, let $[\rho] \in \Bstar{n}{n-k}$ and $X \in \binom{[n]}{n-k-1}$. Then the following chain of logical equivalences holds:
    \begin{align*}
        X \in \coreviso{n}{n-k}([\rho]) &\iff \text{$\rho|_{P^*(X)}$ is antilex} & \text{by definition of $\coreviso{n}{n-k}$}\\
        &\iff \text{$\beta_{n,k}^{-1}([\rho])|_{P([n] \setminus X)}$ is antilex} & \text{by \eqref{equation:packet-copacket-bijection}}\\
        &\iff [n] \setminus X \in (\reviso{n}{k} \circ \beta_{n,k}^{-1})([\rho]) & \text{by definition of $\reviso{n}{k}$}\\
        &\iff X \in (\gamma_{n,k+1} \circ \reviso{n}{k} \circ \beta_{n,k}^{-1})([\rho]) & \text{by definition of $\gamma_{n,k}$.}
    \end{align*}    
    Therefore, $\coreviso{n}{n-k} = \gamma_{n,k+1} \circ \reviso{n}{k} \circ \beta_{n,k}^{-1}$ and the diagram commutes.
\end{proof}

\section{Deletion and Contraction}
Throughout this section, unless otherwise specified, let $k$ and $n$ be fixed integers such that $1 \le k \le n$ and $n \ge 2$. In this section, deletion and contraction are defined on total orders and subsets of $\binom{[n]}{k}$. For subsets of $\binom{[n]}{k}$, the definitions of deletion and contraction agrees with the standard definitions in matroid theory. A new poset $\P_I$ is associated to each consistent set $I \in \C{n}{k}$ and used in \cref{theorem:contraction-deletion-equation} to characterize the possible deletions and contractions of consistent sets.

\begin{definition}
For a set $I \subseteq \binom{[n]}{k}$, the \defn{deletion} of $I$ is
\begin{equation}
    \boldsymbol{I \delete n} = I \cap \binom{[n-1]}{k}.
\end{equation}
The \defn{contraction} of $I$ is
\begin{equation}
    \boldsymbol{I \contract n} = \{X \delete \{n\} : \text{$X \in I$ and $n \in X$}\}.
\end{equation}
\end{definition}

\begin{definition}
For a total order $\rho$ on $\binom{[n]}{k}$, the \defn{deletion} $\boldsymbol{\rho \delete n}$ is the total order obtained by restricting $\rho$ to $\binom{[n-1]}{k}$. The \defn{contraction} $\boldsymbol{\rho \contract n}$ is the total order on $\binom{[n-1]}{k-1}$, where $X < Y$ in $\rho \contract n$ if and only if $X \cup \{n\} < Y \cup \{n\}$ in $\rho$.
\end{definition}

\begin{lemma}
    \label{lemma:closed-under-deletion-contraction}
    The following statements hold.
    \begin{enumerate}[label=(\arabic*)]
        \item For all $\rho \in \A{n}{k}$, $\rho \delete n \in \A{n-1}{k}$ and $\rho \contract n \in \A{n-1}{k-1}$. Furthermore, $\Rev(\rho \delete n) = \Rev(\rho) \delete n$ and $\Rev(\rho \contract n) = \Rev(\rho) \contract n$.
        \item For all $\rho \in \Astar{n}{k}$, $\rho \delete n \in \Astar{n-1}{k}$ and $\rho \contract n \in \Astar{n-1}{k-1}$. Furthermore, $\Corev(\rho \delete n) = \Corev(\rho) \delete n$ and $\Corev(\rho \contract n) = \Corev(\rho) \contract n$.
        \item For all $I \in \C{n}{k}$, $I \delete n \in \C{n-1}{k}$ and $I \contract n \in \C{n-1}{k-1}$.
        \item For all $I \in \Cstar{n}{k}$, $I \delete n \in \Cstar{n-1}{k}$ and $I \contract n \in \Cstar{n-1}{k-1}$.
    \end{enumerate}
\end{lemma}
\begin{proof}
    \textbf{Proof of (1):} Let $\rho \in \A{n}{k}$ and $X \in \binom{[n-1]}{k+1}$. Then $(\rho \delete n)|_{P(X)} = \rho|_{P(X)}$. Since $\rho$ is admissible, $\rho|_{P(X)}$ is the lex or antilex order on $P(X)$, and hence so is $(\rho \delete n)|_{P(X)}$. Thus, $\rho \delete n \in \A{n-1}{k}$ and $\Rev(\rho \delete n) = \Rev(\rho) \delete n$.

    Next, let $Y \in \binom{[n-1]}{k}$ and $Z = Y \cup \{n\}$. If $\rho|_{P(Z)}$ is the lex order $(Z_{k+1}, \ldots, Z_1)$ on $P(Z)$, then $(\rho\contract n)|_{P(Y)}$ is the order $(Z_{k+1,k}, Z_{k+1,k-1}, \ldots, Z_{k+1,1})$. Since $Z_{k+1,i} = Y_i$ for $i \in [k]$, the restriction $(\rho\contract n)|_{P(Y)}$ is the lex order on $P(Y)$. Similarly, if $\rho|_{P(Z)}$ is the antilex order on $P(Z)$, then $(\rho\contract n)|_{P(Y)}$. Thus, $\rho / n \in \A{n-1}{k-1}$ and $\Rev(\rho / n) = \Rev(\rho) / n$.

    \textbf{Proof of (2):} Let $\rho \in \Astar{n}{k}$ and $X \in \binom{[n-1]}{k-1}$. If $\rho|_{P_n^*(X)}$ is the lex order $(X^1, \ldots, X^{n-k+1})$ on $P_n^*(X)$, then $X^{n-k+1} = X \cup \{n\}$ and $(\rho \delete n)|_{P_{n-1}^*(X)}$ is the lex order $(X^1, \ldots, X^{n-k})$ on $P_{n-1}^*(X)$. Similarly, if $\rho|_{P_n^*(X)}$ is the antilex order on $P_n^*(X)$, then $(\rho \delete n)|_{P_{n-1}^*(X)}$ is the antilex order on $P_{n-1}^*(X)$. Thus, $\rho \setminus n \in \Astar{n-1}{k}$ and $\Corev(\rho \setminus n) = \Corev(\rho) \setminus n$.

    Next, let $Y \in \binom{[n-1]}{k-2}$ and $Z = Y \cup \{n\}$. If $\rho|_{P_n^*(Z)}$ is the lex order $(Z^1, \ldots, Z^{n-k+1})$ on $P_n^*(Z)$, then $(\rho \contract n)|_{P_{n-1}^*(Y)}$ is the lex order $(Y^1, \ldots, Y^{n-k+1})$ on $P_{n-1}^*(Y)$. Similarly, if $\rho|_{P_n^*(Z)}$ is the antilex order on $P_n^*(Z)$, then $(\rho \contract n)|_{P_{n-1}^*(Y)}$ is the antilex order on $P_{n-1}^*(Y)$. Thus, $\rho / n \in \Astar{n-1}{k-1}$ and $\Corev(\rho / n) = \Corev(\rho) / n$.

    \textbf{Proof of (3):} Let $I \in \C{n}{k}$ and $X \in \binom{[n-1]}{k+1}$. Then $P(X) \cap (I \delete n) = P(X) \cap I$. Since $I$ is consistent, $P(X) \cap I$ is either a prefix or suffix of $P(X)$ in lex order and hence so is $P(X) \cap (I \delete n)$. Thus, $I \setminus n \in \C{n-1}{k}$.

    Next, let $Y \in \binom{[n-1]}{k}$ and $Z = X \cup \{n\}$. If $P(Z) \cap I$ is a prefix $(Z_{k+1}, \ldots, Z_{i})$, then $P(Y) \cap (I \contract n)$ is a prefix $(Z_{k+1, k}, \ldots, Z_{k+1,i}) = (Y_k, \ldots, Y_i)$. Similarly, if $P(Z) \cap I$ is a suffix of $P(Z)$ in lex order, then $P(Y) \cap (I \contract n)$ is a suffix of $P(Z)$ in lex order. Thus, $I / n \in \C{n-1}{k-1}$.

    \textbf{Proof of (4):} Let $I \in \Cstar{n}{k}$ and $X \in \binom{[n-1]}{k-1}$. If $P_n^*(X) \cap I$ is a prefix $(X^1, \ldots, X^i)$ of $P_n^*(X)$, then $P_{n-1}^*(X) \cap (I \delete n)$ is the prefix $(X^1, \ldots, X^{\min(i, n-k)})$ of $P_{n-1}^*(X)$. Similarly, if $P_n^*(X) \cap I$ is a suffix $(X^{n-k+1}, \ldots, X^i)$ of $P_n^*(X)$, then $P_{n-1}^*(X) \cap (I \delete n)$ is a suffix $(X^{n-k}, \ldots, X^i)$ of $P_{n-1}^*(X)$. Thus, $I \delete n \in \Cstar{n-1}{k}$.

    Next, let $Y \in \binom{[n-1]}{k-2}$ and $Z = Y \cup \{n\}$. If $P_n^*(Z) \cap I$ is a prefix $(Z^1, \ldots, Z^i)$ of $P_n^*(Z)$, then $P_{n-1}^*(Y) \cap (I \contract n)$ is the prefix $(Y^1, \ldots, Y^i)$ of $P_{n-1}^*(Y)$. Similarly, if $P_n^*(Z) \cap I$ is a suffix of $P_n^*(Z)$, then $P_{n-1}^*(Y) \cap (I\contract n)$ is a suffix of $P_{n-1}^*(Y)$. Thus, $I \contract n \in \Cstar{n-1}{k-1}$.
\end{proof}

A consequence of \cref{lemma:closed-under-deletion-contraction} is that the deletion and contraction of (co)commutation classes of (co)admissible orders are well-defined. Thus, denote the deletion of a $[\rho] \in \B{n}{k}$ by $[\rho] \delete n$, the contraction by $[\rho] \contract n$, and similarly for $[\sigma] \in \Bstar{n}{k}$.

\begin{lemma}
    \label{lemma:contraction-deletion-are-dual}
    The following equations hold, for $[\rho] \in \B{n}{k}$ and $I \in \C{n}{k}$.
    \begin{enumerate}[label=(\arabic*)]
        \item $\gamma_{n-1,k-1}(I \contract n) = \gamma_{n,k}(I) \delete n$,
        \item $\gamma_{n-1,k}(I \delete n) = \gamma_{n,k}(I) \contract n$,
        \item $\beta_{n-1,k}([\rho] \delete n) = \beta_{n,k}([\rho]) \contract n$, and
        \item $\beta_{n-1,k-1}([\rho] \contract n) = \beta_{n,k}([\rho]) \delete n$.
    \end{enumerate}
\end{lemma}
\begin{proof}
The proofs of (1) and (2) identical to the standard proofs of the duality of contraction and deletion in matroids. A proof of (3) is given, and the proof of (4) is similar.

\textbf{Proof of (3):} By \cref{theorem:fundamental-duality}, it suffices to show that $\Corev(\beta_{n-1,k}([\rho] \setminus n)) = \Corev(\beta_{n,k}([\rho]) / n)$. The two coreversal sets are equal according to the following chain of equalities:
\begin{align*}
    \Corev(\beta_{n-1,k}([\rho] \setminus n)) &= \gamma_{n-1,k+1}(\Rev([\rho] \setminus n)) &\text{by \cref{theorem:fundamental-duality}}\\
    &= \gamma_{n-1,k+1}(\Rev([\rho]) \setminus n) &\text{by (1) of \cref{lemma:closed-under-deletion-contraction}}\\
    &= \gamma_{n,k+1}(\Rev([\rho]))/n &\text{by (2)}\\
    &= \Corev(\beta_{n,k}([\rho])) / n &\text{by \cref{theorem:fundamental-duality}}\\
    &= \Corev(\beta_{n,k}([\rho]) / n) & \text{by (2) of \cref{lemma:closed-under-deletion-contraction}.}
\end{align*}
\end{proof}

\begin{definition}
    \label{definition:I-consistent-poset}
    For $I \in \C{n}{k}$, the \defn{consistent poset} $\P_I$ is the poset on $\binom{[n]}{k-1}$ whose partial order $\prec_I$ is the transitive closure of the relations
    \begin{enumerate}[label=(\arabic*)]
        \item $X_i \prec_I X_{i+1}$ for all $X \in I$ and $1 \le i < k$.
        \item $X_{i+1} \prec_I X_i$ for all $X \in \binom{[n]}{k} \setminus I$ and $1 \le i < k$.
    \end{enumerate}
\end{definition}

\begin{remark}
    For $I \in \Cstar{n}{k}$, one can define an analogous \defn{dual consistent poset} $\P^*_I$. In the dual consistent poset, relation (1) is replaced by $X^{i+1} \prec_I X^i$ for all $X \in I$ and $1 \le i < n-k$ and relation (2) is replaced by $X^i \prec_I X^{i+1}$ for all $X \in \binom{[n]}{k} \setminus I$ and $1 \le i < n-k$. \cref{lemma:upper-order-ideal-is-consistent} and \cref{theorem:contraction-deletion-equation} below also hold for dual consistent posets.
\end{remark}

 The transitive closure of relations (1) and (2) in \cref{definition:I-consistent-poset} is acyclic because any admissible order $\rho \in \A{n}{k}$ satisfies the relations and is a linear extension of $\P_{\Rev(\rho)}$. In fact, the map $\reviso{n}{k}^{-1}: \C{n}{k+1} \to \B{n}{k}$ sends a consistent set $I \in \C{n}{k+1}$ to the commutation class consisting of all linear extensions of $\P_I$.

\begin{example}
    \label{example:consistent-poset}
    Let $I = \{123, 124\} \in \C{4}{3}$ be the consistent set from \cref{example:total-order-reversal-set}. Then the consistent poset $\P_I$ is the transitive closure of the antilex relations
    \begin{align*}
        23 \prec_I 13 \prec_I 12,\\
        24 \prec_I 14 \prec_I 12,
    \end{align*}
    and the lex relations
    \begin{align*}
        13 \prec_I 14 \prec_I 34,\\
        23 \prec_I 24 \prec_I 34.
    \end{align*}
    The Hasse diagram of $\P_I$ is shown in \cref{figure:consistent-poset-example}. One can check that the upper and lower order ideals of $\P_I$ are consistent sets in $\C{4}{2}$.
    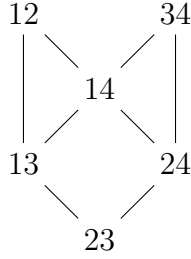
\begin{figure}[!ht]
        \centering
        \begin{tikzpicture}
            \node (a1) at (0, 0) {$23$};
            \node (a2) at (-1, 1) {$13$};
            \node (a3) at (1, 1) {$24$};
            \node (a4) at (0, 2) {$14$};
            \node (a5) at (-1, 3) {$12$};
            \node (a6) at (1, 3) {$34$};
            
            \draw (a1) -- (a2);
            \draw (a1) -- (a3);
            \draw (a2) -- (a4);
            \draw (a3) -- (a4);
            \draw (a2) -- (a5);
            \draw (a3) -- (a6);
            \draw (a4) -- (a5);
            \draw (a4) -- (a6);
        \end{tikzpicture}
        \caption{The Hasse diagram of the consistent poset $\P_{\{123, 124\}}$ in \cref{example:consistent-poset}.}
        \label{figure:consistent-poset-example}
    \end{figure}
\end{example}

\begin{lemma}
    \label{lemma:upper-order-ideal-is-consistent}
    Let $I \in \C{n}{k}$. If $J$ is an upper or lower order ideal of $\P_I$, then $J \in \C{n}{k-1}$.
\end{lemma}
\begin{proof}
    Let $X \in \binom{[n]}{k}$. If $X \in I$, then (1) in \cref{definition:I-consistent-poset} implies that $X_1 \prec_I \cdots \prec_I X_k$. Otherwise, $X \in \binom{[n]}{k} \setminus I$, so (2) in \cref{definition:I-consistent-poset} implies that $X_k \prec_I \cdots \prec_I X_1$. Thus, the ordered sets $X_1, \ldots, X_{k}$ form a chain in either lex or antilex order in $\P_I$. Since $J$ is an upper or lower order ideal of $\P_I$, the intersection $J \cap P(X)$ is either a prefix or suffix of $P(X)$ in lex order. Therefore, $J$ is consistent.
\end{proof}

\begin{lemma}
    \label{theorem:contraction-deletion-equation}
    The map from $\C{n}{k}$ to $\C{n-1}{k} \times \C{n-1}{k-1}$ that sends $I$ to $(I \delete n, I\contract n)$ is injective, and its image is the set 
    \begin{equation*}
        \{(S,T) : \text{$T$ is an upper order ideal of $\P_S$}\}.
    \end{equation*}
\end{lemma}
\begin{proof}
    A consistent set $I \in \C{n}{k}$ can be recovered from $(I \delete n, I \contract n)$ by the equation
    \begin{equation}
        \label{equation:recover-consistent-from-contraction-deletion}
        I = (I \delete n) \cup \{X \cup \{n\} : X \in I \contract n\}.
    \end{equation}
    Therefore, the map $I \mapsto (I \delete n, I \contract n)$ is injective.

    To show that $I \contract n$ is an upper order ideal of the consistent poset $\P_{I \delete n}$, suppose $X = [x_1, \ldots, x_k] \in \binom{[n-1]}{k}$. If $X \in I \delete n$, then $P(X)$ forms a chain in antilex order in $\P_{I \delete n}$. Since $X \in I \delete n \subset I$ and $X$ is the lex smallest element in $P(X \cup \{n\})$, the intersection $P(X \cup \{n\}) \cap I$ is a prefix of $P(X \cup \{n\})$ in lex order. It follows that $P(X) \cap (I \contract n)$ is a prefix of $P(X)$ in lex order. Thus, $P(X) \cap (I \contract n)$ is an upper order ideal of the subposet of $\P_{I \delete n}$ restricted to $P(X)$. A similar argument shows that if $X \not\in I \delete n$, then $P(X) \cap (I \contract n)$ is also an upper order ideal of $P(X)$ in $\P_{I \delete n}$. Therefore, for every $I \in \C{n}{k}$, $I \contract n$ is an upper order ideal of $\P_{I \delete n}$.
    
    Conversely, suppose $(S,T) \in \C{n-1}{k} \times \C{n-1}{k-1}$ and $T$ is an upper order ideal of $\P_S$. Let $I = S \cup \{X \cup \{n\} : X \in T\}$. Then by \eqref{equation:recover-consistent-from-contraction-deletion}, it suffices to show that $I$ is consistent. Let $X = [x_1, \ldots, x_{k+1}] \in \binom{[n]}{k+1}$. If $x_{k+1} < n$, then $P(X) \cap I = P(X) \cap S$. Since $S$ is consistent, $P(X) \cap S$ is a prefix or suffix of $P(X)$ in lex order. If $x_{k+1} = n$ and $X_{k+1} \in S$, then since $T$ is an upper order ideal of $\P_S$, $P(X_{k+1}) \cap T$ is a prefix of $P(X_{k+1})$ in lex order, and hence $P(X) \cap I$ is a prefix of $P(X)$ in lex order. Similarly, if $x_{k+1} = n$ and $X_{k+1} \not\in S$, then $P(X) \cap I$ is a suffix of $P(X)$ in lex order. Thus, $I$ is consistent, completing the proof.
\end{proof}

\begin{remark}
    Some of the ideas in \cref{theorem:contraction-deletion-equation} are present in Lemma 2.5 of \cite{ziegler93}. The new characterization in terms of consistent posets is based on joint work with Billey and Liu \cite{billey-chau-liu} and used in this paper to derive asymptotic results in Section 6.
\end{remark}

\section{Weaving Functions}
\label{section:weaving-functions}
In this section, a new computational tool called weaving functions is introduced. Weaving functions generalize an encoding of elements in $\B{n}{2}$ by Felsner in \cite{felsner1997} to an encoding of elements in $\B{n}{k}$ for integers $k$ and $n$ satisfying $1 \le k \le n$.

Let $\{0,1\}^*$ denote the set of words in the alphabet $\{0,1\}$. The empty word is denoted $\varnothing$, and the concatenation of words $u,v \in \{0,1\}^*$ is denoted $uv$.

\begin{definition}
    For integers $k,n$ with $1 \le k \le n$, the \defn{prefix-suffix indicator function} associated with $Y \in \binom{[n]}{k}$ is the function $\boldsymbol{w_Y}: \binom{[n]}{k-1} \to \{0,1\}^*$ defined by
    \begin{equation}
        w_{Y}(X) = \begin{cases}
            0 & \text{if $X = Y_k$,}\\
            1 & \text{if $X = Y_1$,}\\
            \varnothing & \text{otherwise.}
        \end{cases}
    \end{equation}
    The \defn{weaving function} associated with $\rho = (\rho_1, \ldots, \rho_{\binom{n}{k}}) \in \A{n}{k}$ is the function $\boldsymbol{W_\rho}: \binom{[n]}{k-1} \to \{0,1\}^*$ defined by 
    \begin{equation}
        W_\rho(X) = w_{\rho_1}(X)w_{\rho_2}(X) \cdots w_{\rho_{\binom{n}{k}}}(X).
    \end{equation}
\end{definition}

\begin{example}
    Let $\rho = (23,24,25,45,13,15,35,14,34,12) \in \A{5}{2}$ be the admissible order from \eqref{equation:example-admissible-order}. Then 
    \begin{align*}
        W_\rho(1) &= 0000,\\
        W_\rho(2) &= 0001,\\
        W_\rho(3) &= 0011,\\
        W_\rho(4) &= 1011,\\
        W_\rho(5) &= 1111.
    \end{align*}
    Let $\sigma = (123, 124, 125, 134, 135, 145, 234, 235, 245, 345) \in \A{5}{3}$. Some values of $W_\sigma$ are $W_\sigma(23) = 100$, $W_\sigma(24) = 10$ and $W_\sigma(34) = 110$. All other values of $W_\sigma$ consist solely of zeroes or solely of ones. Observe that all the values of $W_\rho$ are of the same length, but the same is not true of the values of $W_\sigma$.
\end{example}

\begin{lemma}
    \label{lemma:k-2-word-length}
    Let $n$ be an integer with $n \ge 2$. Then for $\rho \in \A{n}{2}$ and $i \in [n]$, $W_\rho(i)$ is a word of length $n-1$.
\end{lemma}
\begin{proof}
    For $[p,q] \in \binom{[n]}{2}$, $w_{[p,q]}(i) \neq \varnothing$ if and only if either $p = i$ or $q = i$. There are exactly $n-1$ sets in $\binom{[n]}{2}$ that contain $i$ so $W_\rho(i)$ is a word of length $n-1$.
\end{proof}

A natural question one can ask is which admissible orders in $\A{n}{k}$ have the same weaving functions. \cref{theorem:weaving-functions} below implies that two admissible orders have the same weaving function if and only if they are commutation equivalent. As a consequence, weaving functions are well-defined on commutation classes in $\B{n}{k}$ and the map $[\rho] \mapsto W_{\rho}$ is injective.

\begin{lemma}
    \label{lem:weaving-fuctions-helper}
    Let $k,n$ be integers such that $1 \le k \le n$ and $\rho = (\rho_1, \ldots, \rho_{\binom{[n]}{k}})$, $\sigma = (\sigma_1, \ldots, \sigma_{\binom{[n]}{k}})$ be two admissible orders in $\A{n}{k}$ such that $W_\rho = W_\sigma$. If $(\rho_1, \ldots, \rho_{i-1}) = (\sigma_1, \ldots, \sigma_{i-1})$ for some integer $i$ such that $1 \le i \le \binom{n}{k}$, then there exists $\sigma' \in \A{n}{k}$ such that $\sigma \sim \sigma'$ and $(\rho_1, \ldots, \rho_i) = (\sigma'_1, \ldots, \sigma'_i)$.
\end{lemma}
\begin{proof}
    Let $j \ge i$ be the index in $\sigma$ such that $\rho_i = \sigma_j$. Suppose for contradiction that $\sigma_j$ does not commute with some element in $\{\sigma_i, \ldots, \sigma_{j-1}\}$. Then let $i'$ be the least integer such that $i \le i' < j$ and $\sigma_{i'}$ does not commute with $\sigma_j$, and let $j'$ be the index in $\rho$ such that $\rho_{j'} = \sigma_{i'}$. Let $X = \rho_i = \sigma_j$ and $Y = \rho_{j'} = \sigma_{i'}$. Since $X$ and $Y$ do not commute, there exists some $Z \in \binom{[n]}{k+1}$ such that $X,Y \in P(Z)$. By interchanging the roles of $\rho$ and $\sigma$ if necessary, one may assume that $X$ is lexicographically smaller than $Y$.

    Observe that $X$ occurs before $Y$ in $\rho$, whereas $Y$ occurs before $X$ in $\sigma$. Since $\rho$ and $\sigma$ are admissible orders, the restrictions $\rho|_{P(Z)}$ and $\sigma|_{P(Z)}$ must be the lex and antilex orders on $P(Z)$, respectively. Since $X = \rho_i$ and $(\rho_1, \ldots, \rho_{i-1}) = (\sigma_1, \ldots, \sigma_{i-1})$, it must be that $X = Z_{k+1}$. Since $Y = \sigma_{i'}$ and $i'$ was chosen to be the least integer between $i$ and $j$ such that $\sigma_{i'}$ does not commute with $\sigma_j$, it must be that $Y = Z_1$.
    
    The word $W_\rho(Z_{1,k+1})$ begins with the prefix $w_{\rho_1}(Z_{1,k+1}) \cdots w_{\rho_{i-1}}(Z_{1,k+1})$ and the word $W_\sigma(Z_{1,k+1})$ begins with the prefix $w_{\sigma_1}(Z_{1,k+1}) \cdots w_{\sigma_{i'-1}}(Z_{1,k+1})$. The choice of $i'$ implies that $\sigma_r \neq Z_{1,k+1} \cup \{z\}$ for any $z \in [n] \setminus Z_{1,k+1}$ and $i \le r < i'$. Therefore, $w_{\sigma_r}(Z_{1,k+1}) = \varnothing$ for $i \le r < i'$, and hence 
    \[
    w_{\rho_1}(Z_{1,k+1}) \cdots w_{\rho_{i-1}}(Z_{1,k+1}) = w_{\sigma_1}(Z_{1,k+1}) \cdots w_{\sigma_{i'-1}}(Z_{1,k+1}).
    \]
    However, $w_{\rho_i}(Z_{1,k+1}) = 1$ and $w_{\sigma_{i'}}(Z_{1,k+1}) = 0$. Therefore, $W_\rho(Z_{1,k+1}) \neq W_\sigma(Z_{1,k+1})$ which is a contradiction. It follows that $\sigma_j$ commutes with every element in $\{\sigma_i, \ldots, \sigma_{j-1}\}$ to yield an admissible order $\sigma' \sim \sigma$ such that $(\rho_1, \ldots, \rho_i) = (\sigma'_1, \ldots, \sigma'_i)$.
\end{proof}

\begin{theorem}
    \label{theorem:weaving-functions}
    For integers $1 \le k \le n$ and $[\rho], [\sigma] \in \B{n}{k}$, $[\rho] = [\sigma]$ if and only if $W_\rho = W_\sigma$.
\end{theorem}
\begin{proof}
    First, suppose $[\rho] = [\sigma]$. It suffices to prove that $W_\rho = W_\sigma$ in the case where $\rho$ and $\sigma$ differ by a single commutation. Let $\sigma$ be obtained from $\rho$ by commuting $X,Y \in \binom{[n]}{k}$. Since $X$ and $Y$ commute, $P(X) \cap P(Y) = \varnothing$. In particular, the sets $\{X_1, X_k\}$ and $\{Y_1, Y_k\}$ are disjoint. Thus, for any $Z \in \binom{[n]}{k-1}$, at least one of $w_X(Z)$ or $w_Y(Z)$ is the empty string $\varnothing$, which implies that $w_X(Z)w_Y(Z) = w_Y(Z)w_X(Z)$. Since $\rho$ and $\sigma$ differ only by commuting $X$ and $Y$, it follows that $W_\rho = W_\sigma$. 

    Next, suppose that $W_\rho = W_\sigma$. To show that $[\rho] = [\sigma]$, let $1 \le i \le \binom{n}{k}$ be the maximum integer such that $(\rho_1, \ldots, \rho_{i-1}) = (\sigma_1, \ldots, \sigma_{i-1})$. If $i = \binom{n}{k}$, then $\rho = \sigma$ and hence $[\rho] = [\sigma]$. Otherwise, let $i < \binom{n}{k}$. Then by \cref{lem:weaving-fuctions-helper}, there exists $\sigma' \in \A{n}{k}$ such that $\sigma' \sim \sigma$ and $(\rho_1, \ldots, \rho_i) = (\sigma'_1, \ldots, \sigma'_i)$. Repeating the argument on $\rho$ and $\sigma'$ implies that $[\rho] = [\sigma']$. Since $\sigma \sim \sigma'$, it follows that $[\rho] = [\sigma]$.
\end{proof}

\section{Asymptotic Enumeration}
In this section, asymptotics are obtained for $|\B{n}{k}|$ and $|\Bstar{n}{k}|$. The Eulerian number counting the number of permutations in $\S_n$ with $d$ descents is denoted $\boldsymbol{A(n,d)}$. The following theorem is the more precise statement of \cref{theorem:higher-bruhat-order-asymptotics-intro}.

\begin{theorem}
    \label{theorem:higher-bruhat-order-asymptotics}
    For every integer $k \ge 2$ and sufficiently large $n \gg k$, we have
    \begin{equation}
    \label{equation:higher-bruhat-order-asymptotics}
    \frac{A(k, \lfloor (k-1)/2 \rfloor) n^k}{k!(k+1)!} + O(n^{k-1}) \le \log_2 |\B{n}{k}| \le \frac{n^k}{k!\log{2}} + O(n^{k-1}\log n).
    \end{equation}
\end{theorem}

\begin{proof}
    The upper bound is proved by induction on $k$. To prove the base case $k = 2$, consider the number of possible weaving patterns. By \cref{lemma:k-2-word-length} and \cref{theorem:weaving-functions}, $|\B{n}{2}|$ is bounded above by the number of functions $f: [n] \to \{0,1\}^*$ such that $f(i)$ is a binary word consisting of $(i-1)$ zeroes and $(n-i)$ ones. There are $\prod_{i=1}^n \binom{n-1}{i-1}$ such functions, which by Theorem 3.2 of \cite{lagariasmehta2016}, has the asymptotic expression
    \[
    \log_2 \left(\prod_{i=1}^n \binom{n-1}{i-1}\right) = \frac{n^2}{2\log{2}} + O(n\log n).
    \]
    
    Suppose the upper bound holds for $\log_2|\B{n}{k-1}|$ by induction. Repeated application of \cref{theorem:contraction-deletion-equation} to $|\C{n}{k+1}|$ yields an upper bound of
    \begin{align*}
    |\C{n}{k+1}| &\le |\C{n-1}{k+1}| \cdot |\C{n-1}{k}|\\
    &\le |\C{n-2}{k+1}| \cdot |\C{n-2}{k}| \cdot |\C{n-1}{k}|\\
    &\qquad\vdots\\
    &\le \prod_{m=k}^{n-1} |\C{m}{k}|.
    \end{align*}
    By \cref{theorem:fundamental-duality}, $|\C{n}{k+1}| = |\B{n}{k}|$ and $|\C{m}{k}| = |\B{m}{k-1}|$ for $m = k, k+1, \ldots, n-1$. Taking logarithms yields
    \begin{align*}
        \log_2 |\B{n}{k}| &\le \sum_{m=k}^{n-1} \log_2 |\B{m}{k-1}|\\
        &\le \sum_{m=k}^{n-1} \left(\frac{m^{k-1}}{(k-1)! \log 2} + O(m^{k-2} \log m)\right)\\
        &\le \frac{n^k}{k!\log 2} + O(n^{k-1}\log n),
    \end{align*}
    where the last step follows from Faulhaber's formula for the sum of the $(k-1)$th powers.
    
    Next, the lower bound is proved by an explicit construction. \cref{lemma:upper-order-ideal-is-consistent} implies that $|\C{n}{k+1}|$ is bounded below by the number of upper order ideals of $\P_\varnothing$. The partial order in $\P_\varnothing$ is equal to the Gale order where $[x_1, \ldots, x_{k+1}] \preceq [y_1, \ldots, y_{k+1}]$ if and only if $x_i \le y_i$ for $i = 1, 2, \ldots, k+1$. The number of upper order ideals is therefore at least $2^{\width(\P_\varnothing)}$ where $\width(\P_\varnothing)$ is the maximal size of an antichain in $\P_\varnothing$. Taking logarithms yields 
    \begin{equation}
        \width(\P_\varnothing) \le \log_2 |\C{n}{k+1}|.
    \end{equation}
    
    To bound $\width(\P_\varnothing)$, for an integer $c$ consider the set 
    \begin{equation}
        S_c = \left\{[x_1, \ldots, x_{k+1}] \in \binom{[n]}{k+1}: \sum_{i=1}^{k+1} x_i = c\right\}.
    \end{equation}
    Observe that $S_c$ is an antichain in $\P_\varnothing$. The elements in $S_c$ are in bijection with partitions of $c-\binom{k+1}{2}$ that fit in an $(k+1) \times (n-(k+1))$ rectangle, with the explicit bijection map
    \begin{equation}
        [x_1, \ldots, x_{k+1}] \leftrightarrow [x_1 - 1, x_2 -2, \ldots, x_{k+1} - (k+1)].
    \end{equation}
    The number of such restricted partitions is known to be given by the coefficient of $q^{c-\binom{k+1}{2}}$ in the $q$-binomial coefficient $\binom{n}{k+1}_q$. Thus, 
    \begin{equation}
        [q^{\lfloor\frac{k+1}{2}\rfloor(n-k-1)}]\binom{n}{k+1}_q \le \max_i\ [q^i]\binom{n}{k+1}_q \le \width(\P_\varnothing),
    \end{equation}
    where the notation $[q^i]\binom{n}{k+1}_q$ denotes the coefficient of $q^i$ in $\binom{n}{k+1}_q$. By Theorem 2.4 \cite{stanleyzanello2016} and the discussion of Euler-Frobenius numbers following the theorem, for a fixed integer $\alpha \ge 0$, the following asymptotic holds for $a \to \infty$
    \begin{equation}
    \label{equation:stanley-zanello}
    [q^{\alpha a}]\binom{a+k+1}{k+1}_q = \frac{A(k,  \alpha-1)a^k}{k!(k+1)!} + O(a^{k-1}).
    \end{equation}
    Notice that in \cite{stanleyzanello2016}, the Eulerian number $A(n,d)$ is defined to by the number of permutations in $\mathfrak{S}_n$ with $d-1$ descents as opposed to $d$ descents. This is accounted for in \eqref{equation:stanley-zanello} by writing $A(k, \alpha-1)$ as opposed to $A(k, \alpha)$. Finally, setting $\alpha = \lfloor \frac{k+1}{2} \rfloor$ and $a = n-k-1$ yields the desired lower bound 
    \[
    \frac{A(k, \lfloor (k-1)/2 \rfloor)n^k}{k!(k+1)!} + O(n^{k-1}) \le \log_2 |\C{n}{k+1}|.
    \]
\end{proof}

\begin{remark}
    The constant $c_k$ in \cref{theorem:higher-bruhat-order-asymptotics-intro} is given by the maximal Eulerian number $A(k, \lfloor (k-1)/2 \rfloor)$ (see  OEIS sequence \href{https://oeis.org/A006551}{A006551}). From Equation 5.7 and 5.8 of \cite{bender1973}, the asymptotic behavior is 
    \begin{equation}
    A(k, \lfloor (k-1)/2 \rfloor) \sim \frac{k!e^{-\alpha \lfloor(k-1)/2\rfloor}}{r^{n+1}\sigma_\alpha\sqrt{2\pi k}},
    \end{equation}
    where
    \begin{equation}
        \label{equation:bender2}
        \begin{split}
        \frac{\lfloor (k-1)/2 \rfloor}{k} &= \frac{e^\alpha}{e^\alpha-1} - \frac{1}{\alpha}\\
        r &= \frac{\alpha}{e^{\alpha}-1}\\
        \sigma_\alpha^2 &= \frac{1}{\alpha^2} - \frac{e^{\alpha}}{(e^{\alpha}-1)^2}.
        \end{split}
    \end{equation}
    As $k \to \infty$, \eqref{equation:bender2} implies $\alpha \to 0$, $r \to 1$, and $\sigma_\alpha^2 \to \frac{1}{12}$. Thus, as $k$ grows large, $c_k$ tends to $\frac{k!}{\sqrt{24 \pi k}}$ and the lower bound for $\log_2 |\B{n}{k}|$ tends to $\displaystyle \frac{n^k}{(k+1)!\sqrt{24 \pi k}}$.
\end{remark}

\begin{theorem}
\label{theorem:dual-higher-bruhat-order-asymptotics}
For integers $k \ge 3$ and $n \gg k$, we have
\begin{equation}
    \label{equation:dual-higher-bruhat-order-asymptotics}
    \frac{A(k-2, \lfloor (k-3)/2 \rfloor)n^{k-2}}{(k-2)!(k-1)!} + O(n^{k-3})
    \le \log_2 |\Bstar{n}{k}| \le \frac{n^{k-2}}{(k-2)!} + O(n^{k-3}\log n).
\end{equation}
\end{theorem}
\begin{proof}
    When $k = 3$, the exact formula $|\B{n}{n-3}| = 2^n + n2^{n-2} - 2n$ is known (see \cref{theorem:codimension3-formula}). By \cref{theorem:fundamental-duality}, $|\Bstar{n}{3}| = |\B{n}{n-3}|$ and taking logarithms yields $\log_2 |\Bstar{n}{3}| = n + O(\log n)$. For $k = 3$, the coefficient on the lower bound is $\frac{A(1,0)}{1!2!} = \frac{1}{2}$ and the coefficient on the upper bound is $1$. Thus, the lower and upper bounds hold for $k = 3$. The upper bound is proved by induction on $k > 3$. Repeatedly applying the dual version of \cref{theorem:contraction-deletion-equation} gives an upper bound of
    \begin{equation}
        \label{equation:dual-upper-bound}
        |\Cstar{n}{k-1}| \le \prod_{m=k-1}^{n-1} |\Cstar{m}{k-2}|.
    \end{equation}
    By \cref{theorem:fundamental-duality}, $|\Cstar{n}{k-1}| = |\Bstar{n}{k}|$ and $|\Cstar{m}{k-2}| = |\Bstar{m}{k-1}|$. Substituting into \cref{equation:dual-upper-bound} and taking logarithms yields
    \begin{align*}
        \log_2 |\Bstar{n}{k}| &\le \sum_{m=k-1}^{n-1} \log_2 |\Bstar{m}{k-1}|\\
        &\le \sum_{m=k-1}^{n-1} \left(\frac{1}{(k-3)!}m^{k-3} + O(m^{k-4} \log m)\right)\\
        &\le \frac{1}{(k-2)!}n^{k-2} + O(n^{k-3} \log n).
    \end{align*}

    Next, \cref{lemma:upper-order-ideal-is-consistent} bounds $|\Cstar{n}{k-1}| = |\Bstar{n}{k}|$ from below by the number of upper order ideals of the dual consistent poset $\P^*_\varnothing$. The partial order on $\P^*_\varnothing$ is also equal to the Gale order, so exactly the same reasoning as in the proof of \cref{theorem:higher-bruhat-order-asymptotics} applies. Therefore, $\frac{A(k-2, \lfloor (k-3)/2\rfloor)}{(k-2)!(k-1)!}n^{k-2} + O(n^{k-3}) \le |\Cstar{n}{k-1}|$.
\end{proof}

\section{Explicit Enumeration}
\label{section:explicit-formulas}
This section supplies a proof of the explicit formula for $|\B{n}{n-3}|$ that was first put forth in \cite{ziegler93}. The proof techniques generalize to enumerate subsets of $\B{n}{n-4}$. In particular, \cref{theorem:TSPP-bijection} is proved.

For an element $I \in \Cstar{n-1}{k-1}$, let $\CstarI{I}{n}{k} = \{J \in \Cstar{n}{k} : J\contract n = I\}$. The sets $\CstarI{I}{n}{k}$ as $I$ ranges over all elements $\Cstar{n-1}{k-1}$ partition the elements of $\Cstar{n}{k}$ by their contraction. It follows that 
\begin{equation}
\label{equation:partition-by-contraction}
|\Cstar{n}{k}| = \sum_{I \in \Cstar{n-1}{k-1}} |\CstarI{I}{n}{k}|.
\end{equation}

By \cref{theorem:fundamental-duality}, $|\B{n}{n-3}| = |\Cstar{n}{2}|$ so to enumerate $\B{n}{n-3}$, it suffices to enumerate $\Cstar{n}{2}$. The contraction of an element in $\Cstar{n}{2}$ is an element in $\Cstar{n-1}{1}$ and one can check that the elements in $\Cstar{n-1}{1}$ are either of the form $\{\{1\}, \ldots, \{r\}\}$ or $\{\{r+1\}, \ldots, \{n-1\}\}$ for some integer $r$ such that $0 \le r \le n-1$. If $r = 0$, then by convention $\{\{1\}, \ldots, \{r\}\} = \varnothing$ and if $r = n-1$, then by convention $\{\{r+1\}, \ldots \{n-1\}\} = \varnothing$. For example, recall that the elements of $\Cstar{4}{1}$ are depicted on the right hand side of \cref{figure:dual-higher-bruhat-orders}.

A subset $I \subseteq \binom{[n]}{2}$ can be visualized as a set of squares in the plane. The square in $\N^2$ with bottom left vertex at $(x_1, x_2)$ and top right vertex at $(x_1+1, x_2+1)$ is associated to each $[x_1,x_2] \in I$. If $I$ is a coconsistent subset in $\Cstar{n}{2}$, then the contraction $I/n$ can be determined from the $x$-coordinates of the squares in the top row of the visualization.

\begin{example}
\label{example:cn2_visualization}
The coconsistent set $\binom{[6]}{2} \in \Cstar{6}{2}$ is depicted by the squares in \cref{figure:example-diagram} and $I = \{14, 15, 16, 23, 24, 25, 26, 34, 35, 36\} \in \Cstar{6}{2}$ is depicted by the shaded squares. The contraction $I/6$ is the element $\{\{1\}, \{2\}, \{3\}\} \in \C{5}{1}$ so $I \in \CstarI{\{\{1\}, \{2\}, \{3\}\}}{6}{2}$.

\begin{figure}[ht]
    \centering
    \input{figures/diagram0}
    \caption{The shaded squares represent the element $I \in \Cstar{6}{2}$ of \cref{example:cn2_visualization}.}
    \label{figure:example-diagram}
\end{figure}
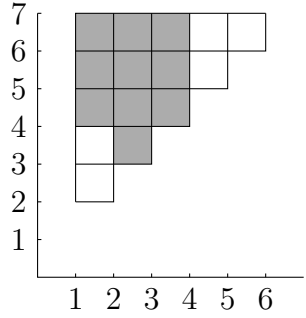
\end{example}

\begin{lemma}
\label{lemma:cn2-refinement-emptyset}
For an integer $n \ge 2$, let $I(n-1) = \{\{1\}, \ldots, \{n-1\}\}$. Then the following formula holds
\begin{equation}
    \label{equation:cn2-refinement-emptyset}
    |\CstarI{\varnothing}{n}{2}| = |\CstarI{I(n-1)}{n}{2}| = 2^{n-2}.
\end{equation}
\end{lemma}
\begin{proof}
    One can check that the map sending $I \mapsto \binom{[n]}{2} \setminus I$ is a bijection between $\CstarI{\varnothing}{n}{2}$ and $\CstarI{I(n-1)}{n}{2}$. Thus, it suffices to prove that \cref{equation:cn2-refinement-emptyset} holds for $|\CstarI{\varnothing}{n}{2}|$. 

    Let $J \in \CstarI{\varnothing}{n}{2}$. The contraction $J/n$ is empty so for any $\{x\} \in \binom{[n-1]}{1}$, $[x,n] \not\in J$. Since the intersection $P^*(\{x\}) \cap J$ does not contain the maximal element $[x,n]$ of $P^*(\{x\})$ and $J$ is coconsistent, the intersection $P^*(\{x\}) \cap J$ is a prefix of $P^*(\{x\})$ in lex order. Thus coconsistency of $J$ imposes two types of convexity conditions:
    \begin{enumerate}
        \item If $[x,y] \in J$ and $1 < x$, then $[x-1, y] \in J$. 
        \item If $[x,y] \in J$ and $x < y-1$, then $[x, y-1] \in J$. 
    \end{enumerate}
    Therefore every coconsistent set $J \in \CstarI{\varnothing}{n}{2}$ is determined by the northeast border of the associated squares. For example, the shaded squares in \cref{figure:c_s_i_empty} depict an element in $\CstarI{\varnothing}{8}{2}$ and the corresponding northeast border.
    
    Mapping an element $J \in \CstarI{\varnothing}{n}{2}$ to the northeast border of the squares associated to $J$ gives a bijection between $\CstarI{\varnothing}{n}{2}$ and lattice paths of length $n-2$ that start from $(1,n)$ and consist only of south and east steps of unit length. There are $2^{n-2}$ such lattice paths, so $|\CstarI{\varnothing}{n}{2}| = 2^{n-2}$.
\end{proof}

\begin{figure}[!ht]
    \centering
    \input{figures/diagram1b.tex}\quad
    \caption{The northeast border of an element in $\CstarI{\varnothing}{8}{2}$.}
    \label{figure:c_s_i_empty}
\end{figure}

\begin{lemma}
    \label{lemma:refinement}
    Let $r,n$ be integers such that $n \ge 2$ and $1 \le r \le n-2$. Let $I(r) = \{\{1\}, \ldots, \{r\}\}$. The following formula holds
    \begin{equation}
        \label{equation:Cn2_refinement}
        |\CstarI{I(r)}{n}{2}| = |\CstarI{I(n-1) \setminus I(r)}{n}{2}| = 2^{n-3} + \binom{n-1}{r} - 1.
    \end{equation}
\end{lemma}
\begin{proof}
    One can check that the map sending $I \mapsto \binom{[n]}{2} \setminus I$ is a bijection between $\CstarI{I(r)}{n}{2}$ and $\CstarI{I(n-1) \setminus I(r)}{n}{2}$. Thus, it suffices to prove that \cref{equation:Cn2_refinement} holds for $|\CstarI{I(r)}{n}{2}|$. The elements in $\CstarI{I(r)}{n}{2}$ can be partitioned into those $J \in \CstarI{I(r)}{n}{2}$ where $[r,r+1] \in J$ and those where $[r,r+1] \not\in J$. The number of elements in each case is computed respectively.

    \textbf{Case 1: $\boldsymbol{[r,r+1] \in J}$.} Since $J/n = \{\{1\}, \ldots, \{r\}\}$, the element $[r,n]$ is in $J$. Because both $[r,r+1], [r,n] \in J$, the coconsistency of $J$ implies that $[r,j] \in J$ for all $j$ satisfying $r+1 \le j \le n$. Furthermore, $[j, n] \not\in J$ for $r+1 \le j < n$, so the intersection $P^*(\{j\}) \cap J$ does not contain the lex maximal element of $P^*(\{j\})$. Thus, $P^*(\{ j \}) \cap J$ is a prefix of the copacket $P^*(\{j\})$ in lex order. It follows that $[i,j] \in J$ for all $1 \le i \le r$ and $r+1 \le j \le n$. Therefore, to enumerate the elements $J \in \CstarI{I(r)}{n}{2}$ that satisfy $[r,r+1] \in J$ amounts to enumerating the possible intersections 
    \begin{align}
        \label{equation:case2a-top-right}
        J &\cap \{[i,j] : r+1 \le i < j \le n\},\\
        \label{equation:case2a-bottom-left}
        J &\cap \{[i,j] : 1 \le i < j \le r\}.
    \end{align}
    The intersections in \cref{equation:case2a-top-right} and \cref{equation:case2a-bottom-left} are outlined in bold in the left diagram of \cref{figure:b_n_i}.
    
    Observe that the possible intersections in \cref{equation:case2a-top-right} are constrained by the coconsistency of $J$ with respect to intersections $J \cap P^*(\{i\})$ where $r+1 \le i \le n$. Similarly, the possible intersections in \cref{equation:case2a-bottom-left} are constrained by the coconsistency of $J$ with respect to intersections $J \cap P^*(\{i\})$ for $1 \le i \le r$. Thus, the intersections in \cref{equation:case2a-top-right} and \cref{equation:case2a-bottom-left} can be chosen independently of each other. Through considering \cref{figure:b_n_i} or the definition of coconsistency, one can check that the possible intersections in \cref{equation:case2a-top-right} are enumerated by the elements in $\CstarI{\varnothing}{n-r}{2}$. By \cref{lemma:cn2-refinement-emptyset}, $|\CstarI{\varnothing}{n-r}{2}| = 2^{n-r-2}$. Similarly, the possible intersections in \cref{equation:case2a-bottom-left} are enumerated by elements in $\CstarI{I(r)}{r+1}{2}$ and by \cref{lemma:cn2-refinement-emptyset}, $|\CstarI{I(r)}{r+1}{2}| = 2^{r-1}$. Therefore, there are $2^{n-r-2} \cdot 2^{r-1} = 2^{n-3}$ elements enumerated in this case.

    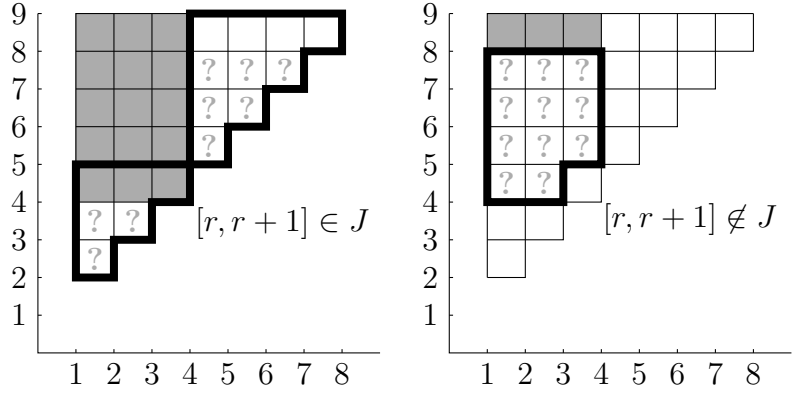
\begin{figure}[!ht]
        \centering
        \input{figures/diagram2}\quad
        \input{figures/diagram3}
        \caption{Case 1 (left) and case 2 (right) in the proof of \cref{lemma:refinement}.}
        \label{figure:b_n_i}
    \end{figure}
    
    \textbf{Case 2: $\boldsymbol{[r,r+1] \not\in J}$.}
    Since $J/n = \{\{1\}, \ldots, \{r\}\}$, $[r+1,n] \not\in J$. Since also $[r,r+1] \not\in J$ by hypothesis, $P^*(\{r+1\}) \cap J$ is a prefix of $P^*(\{r+1\})$ in lex order, so 
    \begin{equation}
        J \cap \{[r+1, i] : r+1 < i \le n\} = \varnothing.
    \end{equation}
    Since $[r+1,i] \not\in J$ for all $i$ satisfying $r+1 < i \le n$, the coconsistency of $J$ implies that $P^*(\{i\}) \cap J$ is a prefix of $P^*(\{i\})$ in lex order. Therefore, $[i,j] \not\in J$ for all $j$ such that $i < j \le n$ and hence
    \begin{equation}
        \label{equation:case2-first-empty}
        J \cap \{[i,j] : r+1 \le i < j \le n\} = \varnothing.
    \end{equation}
    
    Next, $[r,n] \in J$ but $[r,r+1] \not\in J$, so the coconsistency of $J$ implies that $P^*(\{r\}) \cap J$ is a suffix of $P^*(\{r\})$ in lex order and furthermore,
    \begin{equation}
        J \cap \{[j,r] : 1 \le j < r\} = \varnothing.
    \end{equation}
    Therefore, for all $j$ such that $1 \le j < r$, $[j,r] \not\in J$ and $J \in \CstarI{I(r)}{n}{2}$, so we know that $[j,n] \in J$ and  $P^*(\{j\}) \cap J$ is a suffix of $P^*(\{j\})$ in lex order. Thus, $[i,j] \not \in J$ for all $i$ such that $1 \le i < j$, and hence
    \begin{equation}
        \label{equation:case2-second-empty}
        J \cap \{[i,j] : 1 \le i < j \le r\} = \varnothing.
    \end{equation}    
    Both intersections in \cref{equation:case2-first-empty} and \cref{equation:case2-second-empty} are empty. Therefore, to enumerate the sets $J \in \CstarI{I(r)}{n}{2}$ that satisfy $[r,r+1] \not\in J$ amounts to enumerating the possible intersections
    \begin{equation}
        \label{equation:case2b}
        J \cap \{[i,j] : \text{$1 \le i \le r$ and $r+1 \le j < n$}\}.
    \end{equation}
    The intersection in \cref{equation:case2b} is outlined in bold in the right diagram of \cref{figure:b_n_i}.
    
    Let $i,j$ be integers satisfying $1 \le i \le r$ and $r+1 \le j < n$. Observe that the coconsistency of $J$ imposes two types of convexity conditions:
    \begin{enumerate}
        \item If $i > 1$ and $[i,j] \in J$, then $[i-1,j] \in J$.
        \item  If $j < n$ and $[i,j] \in J$, then $[i,j+1] \in J$.
    \end{enumerate}
    Therefore, the sets $J \in \CstarI{I(r)}{n}{2}$ satisfying $[r,r+1] \not\in J$ are in bijection with the  set of lattice paths from $[1,r+1]$ to $[r+1,n]$ consisting of east and north steps of unit length, with the exception of the path consisting of all east steps followed by all north steps. A set $J$ bijects to such a lattice path by mapping $J$ to the southeast border of the squares associated with $J$. There are thus $\binom{n-1}{r}-1$ elements enumerated in this case. Summing the enumeration in cases 1 and 2 together yields $|\CstarI{I(r)}{n}{2}| = 2^{n-3} + \binom{n-1}{r} - 1$.
\end{proof}

\begin{theorem}[{\cite[Proposition 7.1]{ziegler93}}]
    \label{theorem:codimension3-formula}
    For integers $n \ge 4$, $|\B{n}{n-3}| = 2^n + n2^{n-2} - 2n$.
\end{theorem}
\begin{proof}
    By \cref{theorem:fundamental-duality}, $|\B{n}{n-3}| = |\Bstar{n}{3}| = |\Cstar{n}{2}|$. Partitioning coconsistent sets by their contraction implies that 
    \begin{align*}
        |\Cstar{n}{2}| &= \sum_{I \in \Cstar{n-1}{1}} |\CstarI{I}{n}{2}|\\
        &= |\CstarI{\varnothing}{n}{2}| + |\CstarI{I(n-1)}{n}{2}| + \sum_{r=1}^{n-2} \left(|\CstarI{I(r)}{n}{2}| + |\CstarI{I(n-1) \setminus I(r)}{n}{2}|\right)\\
        &= 2|\CstarI{\varnothing}{n}{2}| + 2\sum_{r=1}^{n-2} |\CstarI{I(r)}{n}{2}|,
    \end{align*}
    where the last equality follows by \cref{lemma:cn2-refinement-emptyset} and \cref{lemma:refinement}. By using the explicit formulas \cref{equation:cn2-refinement-emptyset} and \cref{equation:Cn2_refinement}, the summation on the right hand side can be computed as follows
    \begin{align*}
        |\CstarI{\varnothing}{n}{2}| + \sum_{r=1}^{n-2} |\CstarI{I(r)}{n}{2}|
        &= 2^{n-2} + \sum_{r=1}^{n-2} \left(2^{n-3} + \binom{n-1}{r} - 1\right)\\
        &= 2^{n-2} + (n-2)2^{n-3} + (2^{n-1}-2) - (n-2)\\
        &= 2^{n-1} + n2^{n-3} - n.
    \end{align*}
    Multiplying by 2 yields the desired formula.
\end{proof}

\begin{proof}[Proof of \cref{theorem:TSPP-bijection}]
By \cref{theorem:fundamental-duality} and     \cref{lemma:contraction-deletion-are-dual}, the subposet of $\B{n}{n-4}$ obtained by restricting to elements whose deletion is the commutation class of the lexicographic order is isomorphic to the subposet of $\Cstar{n}{3}$ obtained by restricting to $\CstarI{\varnothing}{n}{3}$. By Theorem 4.5 in \cite{ziegler93}, single step inclusion and inclusion coincide for $\C{n}{n-3}$ and hence for $\Cstar{n}{3}$ and $\CstarI{\varnothing}{n}{3}$.

Consider the map $\varphi: \T_{n-3} \to \CstarI{\varnothing}{n}{3}$ defined by
\begin{equation}
\varphi(T) = \{[x_1,x_2+1,x_3+2] : [x_1,x_2,x_3] \in T\}.
\end{equation}
To see that the image of $\varphi$ is $\CstarI{\varnothing}{n}{3}$, let $T \in \T_{n-3}$. For any $[x_1,x_2,x_3] \in T$, the inequalities $1 \le x_1 < x_2 < x_3 \le n-3$ holds. In particular, $x_3 +2 < n$. Therefore, $\varphi(T) / n = \varnothing$. Since $T$ is a plane partition, for any $1 \le i < j \le n$, the intersection $\varphi(T) \cap P^*([i,j])$ is a prefix of the copacket $P^*([i,j])$ in lex order. Thus, $\varphi(T)$ is coconsistent. The inverse map $\varphi^{-1}: \CstarI{\varnothing}{n}{3} \to \T_{n-3}$ is defined to be
\begin{equation}
    \varphi^{-1}(I) = \{[x_{\pi(1)}, x_{\pi(2)}, x_{\pi(3)}] : \text{$[x_1,x_2,x_3] \in I$ and $\pi \in \S_3$}\}.
\end{equation}
It is clear that for any $T, T' \in \T_{n-3}$, the inclusion $T \subseteq T'$ holds if and only if the inclusion $\varphi(T) \subseteq \varphi(T')$ holds. Thus, $\varphi$ is an isomorphism between $\T_{n-3}$ and $\CstarI{\varnothing}{n}{3}$.
\end{proof}

\section{Future Work}
We conclude with some directions for future study. Consistent posets were introduced to obtain a lower bound in \cref{theorem:higher-bruhat-order-asymptotics-intro} but are an interesting family of posets to study in their own right. We have verified the following conjecture on consistent posets $\P_S$ for all $S \in \binom{[n]}{k-1}$ where $n \le 7$ and $2 \le k \le n$.

\begin{conjecture}
    For integers $1 \le k \le n$ and $S \in \C{n}{k}$, the consistent poset $\P_S$ is ranked. 
\end{conjecture}

The weaving functions introduced in \cref{section:weaving-functions} are also interesting objects of study. It is surprising that commutation classes in $\B{n}{k}$ can be recovered from a weaving function when much of the information in a commutation class appears to be ``forgotten'' by considering only binary words.

\begin{problem}
    For integers $1 \le k \le n$, find a combinatorial criterion to determine which functions $W: \binom{[n]}{k-1} \to \{0,1\}^*$ are the weaving function of some $[\rho] \in \B{n}{k}$.
\end{problem}

Extensive computational data has been generated in the case $k = 2$ and is available as part of the \href{https://github.com/pnnl/ML4AlgComb/tree/master}{Algebraic Combinatorics Dataset Repository}. In particular, a transformer machine learning model trained on a subset of the weaving functions of $\B{7}{2}$ attains $99.1\%$ accuracy in identifying whether or not an arbitrary map $[n] \to \{0,1\}^*$ is a weaving function $W_\rho$ for some $[\rho] \in \B{7}{2}$. The parameter size of the trained model is significantly smaller than $|\B{7}{2}|$, suggesting that there is some simple learned embedding or heuristic in the model rather than memorization of the training data.

\cref{theorem:TSPP-bijection} demonstrates that an explicit formula for $|\CstarI{\varnothing}{n}{3}|$ exists and is given by Stembridge's product formula for TSPPs. If one were able to obtain an explicit formula for $|\CstarI{I}{n}{3}|$ for general coconsistent sets $I$, then by \cref{equation:partition-by-contraction}, one may obtain a formula for $\Cstar{n}{3}$.

\begin{problem}
    For a coconsistent set $I \in \Cstar{n-1}{2}$, find a formula for $|\CstarI{I}{n}{3}|$ analogous to the formula in \cref{equation:Cn2_refinement}.
\end{problem}

\subsection*{Acknowledgements}
Many thanks to Sara Billey and Ben Elias for helpful discussions and introducing me to the higher Bruhat orders. Thanks to Kevin Liu for discussing weaving functions with me. Additional thanks to Sean Grate, Helen Jenne, Jamie Kimble, Henry Kvinge, Cordelia Li, Clare Minnerath, Michael Tang and Rachel Wei for feedback.

\newpage
\emergencystretch=1em
\printbibliography

\end{document}

%% file: figures/diagram0.tex
\begin{tikzpicture}[scale=0.5]
\draw (0,7) -- (0,0) -- (7,0);
\foreach \y in {1,2,...,7} {
    \node[anchor=east] at (0, \y) {$\y$};
    \draw (0, \y) -- (0.1, \y);
}
\foreach \x in {1,2,...,6} {
    \draw (\x, 0) -- (\x, 0.1);
    \node[anchor=north] at (\x, 0) {$\x$};
}

\fill[mygray] (1,7)  -- (4,7) -- (4,4) -- (3,4) -- (3,3) -- (2,3) -- (2,4) -- (1,4) -- cycle;

\foreach \i in {2,3,...,6} {
    \draw (1, \i) -- (\i, \i) -- (\i, 7);
}
\draw (1, 2) -- (1, 7) -- (6, 7);


\end{tikzpicture}

%% file: figures/diagram1b.tex
\begin{tikzpicture}[scale=0.5]
\draw (0,9) -- (0,0) -- (9,0);
\foreach \y in {1,2,...,9} {
    \node[anchor=east] at (0, \y) {$\y$};
    \draw (0, \y) -- (0.1, \y);
}
\foreach \x in {1,2,...,8} {
    \draw (\x, 0) -- (\x, 0.1);
    \node[anchor=north] at (\x, 0) {$\x$};
}

\fill[mygray] (1,8) -- (3,8) -- (3,7) -- (4,7) -- (4,4) -- (3,4) -- (3,3) -- (2,3) -- (2,2) -- (1,2) -- cycle;

\foreach \i in {2,3,...,8} {
    \draw (1, \i) -- (\i, \i) -- (\i, 9);
}
\draw (1, 2) -- (1, 9) -- (8, 9);

\draw[line width=3pt] (1,8) -- (3,8) -- (3,7) -- (4,7) -- (4,5);
\fill (1,8) circle (6pt);
\fill (4,5) circle (6pt);

\end{tikzpicture}

%% file: figures/diagram2.tex
\begin{tikzpicture}[scale=0.5]
\draw (0,9) -- (0,0) -- (9,0);
\foreach \y in {1,2,...,9} {
    \node[anchor=east] at (0, \y) {$\y$};
    \draw (0, \y) -- (0.1, \y);
}
\foreach \x in {1,2,...,8} {
    \draw (\x, 0) -- (\x, 0.1);
    \node[anchor=north] at (\x, 0) {$\x$};
}

\fill[mygray] (1,9) -- (4,9) -- (4,4) -- (1,4) -- cycle;
\foreach \x in {4,5,6} {
    \node[mygray] at (\x+.5, 7.5) {\textbf{?}};
}
\node[mygray] at (4.5, 6.5) {\textbf{?}};
\node[mygray] at (5.5, 6.5) {\textbf{?}};
\node[mygray] at (4.5, 5.5) {\textbf{?}};
\node[mygray] at (1.5, 3.5) {\textbf{?}};
\node[mygray] at (2.5, 3.5) {\textbf{?}};
\node[mygray] at (1.5, 2.5) {\textbf{?}};

\foreach \i in {2,3,...,8} {
    \draw (1, \i) -- (\i, \i) -- (\i, 9);
}
\draw (1, 2) -- (1, 9) -- (8, 9);

\draw[line width=3pt] (4,9) -- (8,9) -- (8,8) -- (7,8) -- (7,7) -- (6,7) -- (6,6) -- (5,6) -- (5,5) -- (4,5) -- cycle;
\draw[line width=3pt] (1,5) -- (4,5) -- (4,4) -- (3,4) -- (3,3) -- (2,3) -- (2,2) -- (1,2) -- cycle;

\node[anchor=north west] at (3.85, 4.15) { $[r,r+1] \in J$};

\end{tikzpicture}

%% file: figures/diagram3.tex
\begin{tikzpicture}[scale=0.5]
\draw (0,9) -- (0,0) -- (9,0);
\foreach \y in {1,2,...,9} {
    \node[anchor=east] at (0, \y) {$\y$};
    \draw (0, \y) -- (0.1, \y);
}
\foreach \x in {1,2,...,8} {
    \draw (\x, 0) -- (\x, 0.1);
    \node[anchor=north] at (\x, 0) {$\x$};
}

\fill[mygray] (1,9) -- (4,9) -- (4,8) -- (1,8) -- cycle;
\foreach \x in {1,2,3} {
    \foreach \y in {5,6,7} {
        \node[mygray] at (\x+.5, \y+.5) {\textbf{?}};
    }
}
\node[mygray] at (1.5, 4.5) {\textbf{?}};
\node[mygray] at (2.5, 4.5) {\textbf{?}};

\node[anchor=north west] at (3.75, 4.25) {$[r,r+1] \not\in J$};

\foreach \i in {2,3,...,8} {
    \draw (1, \i) -- (\i, \i) -- (\i, 9);
}
\draw (1, 2) -- (1, 9) -- (8, 9);

\draw[line width=3pt] (1,8) -- (4,8) -- (4,5) -- (3,5) -- (3,4) -- (1,4) -- cycle;

\end{tikzpicture}